\documentclass[11pt,final]{article}

\usepackage{amsmath}
\usepackage{amsthm}
\usepackage{enumitem}        
\usepackage{float}
\usepackage{graphicx}
\usepackage[justification=centering]{caption}
\usepackage{mathtools}
\usepackage[ruled]{algorithm2e}

\usepackage{geometry}
\usepackage{authblk}
       
\usepackage[cmintegrals]{newpxmath}  

\usepackage{bm}                      

\usepackage[style=numeric-comp,%
            abbreviate=true,   
            useprefix=true,    %
            giveninits=false,  %
            hyperref=false,    %
            maxcitenames=2,    %
            maxbibnames=20,    %
            uniquename=init,   %
            sortcites=true,    
            doi=true,          %
            isbn=true,         %
            url=false,         %
            eprint=false,      %
            backend=bibtex     %
           ]{biblatex}         
  \renewbibmacro{in:}{}
  \DefineBibliographyStrings{english}{bibliography = {References}}
  
\usepackage{xcolor}

\usepackage{dsfont}                                              
  \newcommand{\IQ}{\ensuremath\mathds{Q}}                        
  \newcommand{\IR}{\ensuremath\mathds{R}}                        
  \newcommand{\IP}{\ensuremath\mathds{P}}                        

\newcommand*{\setE}{\ensuremath{\mathcal{T}}}                    
\newcommand*{\Gammah}{\Gamma_h}                                  
\newcommand*{\Nst}{{N_\mathrm{T}}}                               
\newcommand*{\Nel}{{N_\mathrm{el}}}                              
\newcommand*{\Nloc}{{N_\mathrm{loc}}}                            
\renewcommand*{\vec}[1]{{\boldsymbol{#1}}}                       
\DeclareMathAlphabet{\mathbfsf}{\encodingdefault}{\sfdefault}{bx}{n}
\newcommand*{\vecc}[1]{\mathbfsf{#1}}                            
\newcommand*{\transpose}[1]{{#1}^\mathrm{T}}                     
\newcommand*{\normal}{\vec{n}}                                   
\newcommand*{\dd}{\mathrm{d}}                                    
\newcommand*{\grad}{\vec{\nabla}}                                
\renewcommand*{\div}{\vec{\nabla}\cdot}                          
\newcommand*{\laplace}{\upDelta}                                 
\newcommand*{\llbrace}{\lbrace\hspace*{-0.18em}\vert}
\newcommand*{\rrbrace}{\vert\hspace*{-0.18em}\rbrace}
\newcommand*{\avg}[1]{\llbrace{#1}\rrbrace}                      
\newcommand*{\jump}[1]{\left\llbracket{#1}\right\rrbracket}      
\newcommand*{\abs}[1]{\ensuremath{|#1|}}                         
\newcommand*{\norm}[2]{\|#1\|_{#2}}                              
\newcommand*{\on}[2]{\left.#1\right\vert_{#2}}                   
\newcommand{\upwind}[1]{#1^{\uparrow}}                           
\newcommand*{\prox}{\mathtt{prox}}                               

\usepackage{longtable}
\usepackage{multirow}
\usepackage{tabularx}                                            
  \newcolumntype{R}{>{\raggedleft\arraybackslash}X}
  \newcolumntype{L}{>{\raggedright\arraybackslash}X}
  \newcolumntype{C}{>{\centering\arraybackslash}X}
\usepackage{booktabs}                                            
\usepackage{rotating}                                            

\usepackage{caption}
\newtheorem{lemma}{Lemma}
\newtheorem{theorem}{Theorem}
\newtheorem{remark}{Remark}

\title{Optimization-based Cell Average Limiters as Metric Projections: Constrained Evolution and Convergence Analysis for Convection-diffusion Equations}
\author[a]{Seungmin Lee}
\author[a]{Sanghyun Lee}
\author[b]{Chen Liu}
\affil[a]{Department of Mathematics, Florida State University, Tallahassee, Florida 32306, USA.}
\affil[b]{Department of Mathematical Sciences, University of Arkansas, Fayetteville, Arkansas 72701, USA.}
\date{\today}

\begin{document}

\maketitle

\hrule\vspace{0.5\baselineskip}
\textbf{Abstract}.\quad
The optimization-based cell average limiter enforces global conservation and bounds, and has been successfully applied to discontinuous Galerkin (DG) discretizations of a wide range of nonlinear systems, from phase-field models to the compressible Navier--Stokes equations. 
In this paper, taking the heat and convection-diffusion equations as prototypes, we analyze the convergence of the scheme resulting from cell average limiting. 
Encoding conservation and admissibility in a differential inclusion identifies the limiter as a backward Euler proximal step in a Lie splitting of the constrained evolution.
This proximal step is the metric projection onto the convex admissible set, and we prove that it is nonexpansive and stable in the $L^2$ norm, and that it does not increase the total variation of the cell averages. Combined with the standard error analysis, nonexpansiveness yields an $\mathcal{O}(\tau+h^{k+1})$ estimate for implicit Euler time stepping with DG discretization of polynomial degree $k$ under a reverse CFL condition, together with an unconditional $\mathcal{O}(\tau+h^{k})$ estimate.
Numerical experiments confirm the predicted convergence rates and the enforcement of the bounds. \\

\vspace{0.5\baselineskip}\hrule

\section{Introduction}
Conservation and admissibility are fundamental properties of many time-dependent partial differential equations (PDEs) in science and engineering. Parabolic equations such as the heat and convection-diffusion equations satisfy the maximum principle~\cite{evans2022partial}. Hyperbolic conservation laws such as the compressible Euler equations preserve invariant domains including the positivity of density and pressure~\cite{guermond2016invariant}. 
Discontinuous Galerkin (DG) methods~\cite{cockburn2012discontinuous,shu2014discontinuous} are among the most popular techniques in scientific computing, offering local conservation, high-order accuracy, and excellent stability and flexibility. However, high-order DG methods do not enjoy discrete admissibility in general, which may lead to unphysical solutions and instability for demanding simulations. 
Limiters that enforce conservation and bounds after each time step are therefore essential for the robustness of high-order DG simulations.

\par
The cell average, namely the zeroth-order moment of a DG solution, is the essential building block of the most successful bound-preserving approaches: once the cell average belongs to the admissible set, pointwise bound-preservation is recovered as a second step. 
One of the most widely used systematic approaches for explicit methods is due to Zhang and Shu~\cite{zhang2010maximum,zhang2010positivity}: under a bound-preserving CFL, which is a sufficient condition but more restrictive than the linear stability CFL and cannot in general be relaxed to it, cell averages of the DG solution are convex combinations of special quadrature point values in the admissible set of the previous time step or Runge--Kutta stage, thus preserving the bounds on cell averages. A scaling limiter then modifies the high-order moments to recover pointwise bounds.
Flux limiting methods~\cite{guermond2016invariant,guermond2019invariant} take a convex combination of a high-order scheme and a low-order bound-preserving scheme, ensuring the discrete solution lies in the admissible set. However, a suitable low-order, i.e., first-order, bound-preserving scheme is not always available or easy to construct. 
Conservative redistribution limiting dispenses with the low-order scheme entirely. Degrees of freedom that violate the bounds borrow from admissible neighbors, with a locally balanced budget that preserves the total mass~\cite{guermond2026locally}.
For a comprehensive review of bound-preserving techniques for DG methods, we refer to~\cite{LHTZ2024FP,wu2025high} and the references therein. 

\par
In recent years, optimization-based approaches have become popular.
Bochev et al. formulate conservative remap and transport as the minimization of the distance to a high-order target subject to inequality constraints derived from physical bounds. They show that flux-corrected schemes arise by shrinking the feasible set~\cite{bochev2011formulation,bochev2013fast,bochev2014optimization,bochev2020optimization}.
Barrenechea et al. analyze algebraically stabilized finite element methods~\cite{barrenechea2016analysis}, and propose a nodally bound-preserving finite element method whose solution is the orthogonal projection of the exact solution onto the closed convex set of bound-preserving finite element functions~\cite{barrenechea2024nodally}.
Van der Vegt et al. couple positivity constraints to implicit DG discretizations through a KKT system solved by an active set semismooth Newton method~\cite{van2019positivity}.
Yan and Xia impose cell average bounds through a KKT limiter in implicit BDF-LDG discretizations and establish optimal error estimates~\cite{yan2025analysis}. Their constraints are coupled into the implicit system. Each step is one nonlinear solve rather than postprocessing and no conservation constraint is imposed.
For a non-exclusive list of further optimization-related algorithms, we refer to~\cite{guba2014optimization,ruppenthal2023optimal,kirby2024high,keith2026proximal}.

\par
An optimization-based cell average limiter seeks the piecewise constant polynomial that minimizes the $L^2$ distance to the cell average of the high-order DG solution, subject to the conservation and invariant-domain-preserving constraints. This limiter was introduced in~\cite{liu2024simple} for phase-field multiphase flow and subsequently extended to Fokker--Planck~\cite{LHTZ2024FP}, compressible Navier--Stokes~\cite{liu2024optimization}, compressible Euler~\cite{liu2025efficient}, and magnetohydrodynamics equations~\cite{liu2026efficient}. The limiter has also been implemented as a postprocessing module in the open-source framework Trixi.jl~\cite{ranocha2022adaptive,schlottkelakemper2021purely,schlottkelakemper2025trixi}.
The available accuracy results concern a single application. In a time-dependent computation the modifications are applied at every time step, and their cumulative effect over the whole simulation is a different question. Despite the practical success of the limiter, we are not aware of a convergence analysis that accounts for its repeated application.

\par
The optimization-based cell average limiter is broadly applicable. It is independent of the underlying PDE. The only problem-specific ingredient is the projection onto the admissible set, while the conservation constraint and the optimization solver are unchanged. It is also independent of the time discretization, since it acts as a postprocessing step on the computed solution. Explicit strong-stability-preserving (SSP) methods may then run at the linear stability CFL, non-SSP explicit methods are covered, and for semi-implicit or implicit schemes admissibility is enforced where neither a bound-preserving CFL nor a monotone high-order discretization can be constructed.
It is practical for large-scale simulations. The minimization problem arising from the cell average limiter has one unknown per element, an admissible-set constraint that decouples elementwise, and a single linear constraint for conservation. It therefore admits efficient splitting methods of $\mathcal{O}(\Nel)$ complexity, where $\Nel$ is the total number of mesh elements. Its cost is negligible compared with that of advancing the PDE~\cite{liu2024simple,liu2024optimization,LHTZ2024FP,liu2025efficient}. Enforcing bound-preservation directly at all degrees of freedom would instead give a minimization with $\Nel\Nloc$ unknowns, where $\Nloc$ is the number of degrees of freedom per element, in which the admissible-set constraint imposes one inequality at every quadrature point on the $\Nloc$ local coefficients, so that even the elementwise projection becomes a quadratic program.
Since the cell average is in addition the quantity on which the Zhang--Shu scaling limiter rests, this paper focuses exclusively on the cell average limiter.

\par
We introduce a differential inclusion framework that provides an analytical foundation for the error analysis. For an equation $\partial_t u + \mathcal{A}u = 0$ with admissible set $G$, we define the constraint set
\begin{align*}
A = \{v \in L^2(\Omega):~ (v, 1) = (u^0, 1) ~~\text{and}~~ v(\vec{x})\in G,~ \forall\vec{x}\in\Omega\}
\end{align*}
and rewrite the PDE as the differential inclusion $\partial_t u + \mathcal{A}u + \partial\iota_A(u) \ni 0$, where $\iota_A$ is the indicator function of $A$ and $\partial\iota_A$ is its subgradient. 
Since the exact solution satisfies $u^\mathrm{exact}(t) \in A$ for all $t \in [0,T]$, one has $0 \in \partial\iota_A(u^\mathrm{exact}(t))$, so $u^\mathrm{exact}$ is also a solution of the differential inclusion. The two formulations are equivalent whenever the differential inclusion admits a unique solution. For the parabolic equations studied in this paper, the equivalence is ensured by the maximum principle, see Section~\ref{sec:equivalent_DI}.
This framework gives a natural operator splitting: the DG step advances the PDE subproblem $\partial_t u + \mathcal{A}u = 0$, and the cell average limiter enforces the constraint subproblem $\partial_t u \in -\partial\iota_A(u)$ via the backward Euler proximal step. The cell average limiter is thus identified as a backward Euler / $\IP_0$ finite volume discretization of the subgradient flow $\partial_t u \in -\partial\iota_A(u)$, with the high-order moments of the DG polynomial carried over unchanged. 
For implicit time integration of parabolic PDEs such as the heat equation or the convection-diffusion equation, the sequential two-stage scheme is a Lie splitting of the differential inclusion:
\begin{align*}
(I + \tau\mathcal{A}) u^{n+1} = U^n \quad\text{and}\quad 
U^{n+1} = \prox^\tau_{\iota_A}(u^{n+1}).
\end{align*}
For the explicit Euler time integration, the associated scheme is an IMEX discretization of the differential inclusion:
\begin{align*}
U^{n+1} = \prox^\tau_{\iota_A} (U^n - \tau \mathcal{A}(U^n)).
\end{align*}
While the differential inclusion framework applies broadly, the rigorous error analysis in this paper is carried out for the heat equation and the convection-diffusion equation. 
This is a deliberate choice. Our aim is not a new high-order bound-preserving method, but an understanding of the cell average limiter itself, and the simplest model together with the simplest time integrator is what makes its properties visible.

\par
The main contributions of this paper are as follows. First, we identify the optimization-based cell average limiter as the metric projection onto the discrete admissible set. Encoding conservation and admissibility in a differential inclusion further identifies this projection as the backward Euler proximal step in a Lie splitting of the constrained evolution.
Second, we prove that the cell average limiter is nonexpansive in the $L^2$ norm. As a direct consequence, the limiter preserves $L^2$ accuracy and $L^2$ stability when applied as a postprocessing step to a DG discretization of a steady elliptic problem, see Remark~\ref{rem:steady_elliptic}. We also show that the limiter does not increase the total variation of the cell averages, see Theorem~\ref{thm:tvdm}.
Third, combining nonexpansiveness with the standard DG error analysis, we establish the following error estimates. For both the heat equation and the convection-diffusion equation, the fully discrete scheme, comprising implicit Euler time stepping, interior penalty DG spatial discretization of polynomial degree $k$, and the optimization-based cell average limiter, satisfies
\begin{align*}
\max_{0 \leq n \leq \Nst} \norm{U_h^n - u^\mathrm{exact}(\cdot, t^n)}{L^2(\Omega)} \leq C(\tau + h^{k+1}),
\end{align*}
under a reverse CFL condition $\tau \geq c h^{2(k+1)/3}$, and converges unconditionally at the suboptimal rate $\mathcal{O}(\tau + h^{k})$. 
The cost of repeated limiting is isolated in a single quantity, the cell average of the elliptic projection error, charged once per time step, see Theorem~\ref{thm:Theta_bound}. Controlling its accumulation is what produces the reverse CFL condition, which is not needed for convergence but buys the additional power of $h$.
To put this in context, enforcing bounds via implicit Euler time stepping combined with matrix monotonicity typically requires $\tau \geq c h^2$ regardless of polynomial degree~\cite{hu2023positivity,shen2022discrete,liu2024structure}.
Our condition $\tau \geq c h^{2(k+1)/3}$ coincides with this at $k = 2$ and is progressively more permissive for higher-order methods ($k \geq 3$), i.e., higher spatial accuracy allows smaller time steps without sacrificing optimal convergence.

\par
The remainder of the paper is organized as follows. Section~\ref{sec:PDE_and_DI} introduces the constrained evolution framework, the differential inclusion, and the operator splitting scheme. Section~\ref{sec:cell_avg_limiter} establishes properties of the cell average limiter, including nonexpansiveness. Section~\ref{sec:error_analysis} derives conditional optimal error estimates for the heat equation and the convection-diffusion equation. An unconditional suboptimal error estimate is proved in Appendix~\ref{app:unconditional_subopt}. Section~\ref{sec:numerical_experiments} presents numerical experiments confirming the theoretical convergence rates and demonstrating the bound-preserving property. Section~\ref{sec:conclusion} draws conclusions.

\section{Constrained evolution and discretization}\label{sec:PDE_and_DI}

\subsection{Encoding admissibility in a differential inclusion}\label{sec:equivalent_DI}
For many PDEs in science and engineering, the physically meaningful solutions satisfy global conservation and belong to a convex admissible set $G$. 
Let us consider a well-posed and conservative time-dependent PDE with appropriate initial data $u^0$ and boundary conditions on a convex bounded polygonal domain $\Omega \subset \IR^d$ over the time interval $[0,T]$, as follows
\begin{align}\label{eq:abstract_pde}
\partial_t u + \mathcal{A}u = 0,
\end{align}
where $u$ is the unknown and $\mathcal{A}$ denotes a spatial differential operator. Let $(\cdot,\cdot)$ denote the $L^2$ inner product over domain $\Omega$. We say model \eqref{eq:abstract_pde} is well-posed if it admits a unique solution, conservative if $(u(t),1) = (u^0,1)$ for all $t\in[0,T]$, and invariant-domain-preserving if $u^0(\vec{x})\in G$ for all $\vec{x}\in\Omega$ implies $u(\vec{x},t)\in G$ for all $\vec{x}\in\Omega$ and $t\in[0,T]$.
\par 
Define the indicator function $\iota_{A}$ of a set $A$ as: $\iota_{A}(u) = 0$, if $u \in A$, otherwise $\iota_{A}(u) = +\infty$. Let $\partial\iota_{A}$ be the subgradient of $\iota_{A}$, namely
\begin{align}\label{eq:def_subgradient}
\xi \in \partial\iota_{A}(u) \quad\Leftrightarrow\quad 
u \in A \quad\text{and}\quad
(\xi, w - u) \leq 0 \quad \forall w \in A.
\end{align}
Notice that $0 \in \partial\iota_{A}(u)$ when $u \in A$.
We consider a constrained evolution approach by encoding conservation and admissibility in a differential inclusion. Define the set
\begin{align}\label{eq:set_A}
A = \{v \in L^2(\Omega):~ (v, 1) = (u^0, 1) ~~\text{and}~~ v(\vec{x})\in G,~ \forall\vec{x}\in\Omega\}.
\end{align}
Since the admissible set $G$ is convex, the set $A$ is convex. In addition, any $v\in A$ has the same total mass as $u^0$ and satisfies $v(\vec{x})\in G$ for all $\vec{x}\in\Omega$.
With this constraint set, we define the following differential inclusion
\begin{align}\label{eq:abstract_pde_modified}
\partial_t u + \mathcal{A}{u} + \partial\iota_{A}(u) \ni 0.
\end{align}
Since model \eqref{eq:abstract_pde} is conservative and invariant-domain-preserving, its exact solution $u^\mathrm{exact}$ satisfies $u^\mathrm{exact}(t)\in A$ for all $t\in[0,T]$. Hence $0\in\partial\iota_{A}(u^\mathrm{exact}(t))$ and $u^\mathrm{exact}$ always satisfies \eqref{eq:abstract_pde_modified}.
Provided \eqref{eq:abstract_pde_modified} admits a unique solution, the two formulations are equivalent. Well-posedness of \eqref{eq:abstract_pde_modified} depends on the properties of $\mathcal{A}$ and $G$.
Within this abstract framework, in the rest of this paper, we focus on the diffusion operator $\mathcal{A} = -\laplace$ and the convection-diffusion operator $\mathcal{A} = -D\laplace + \vec{b}\cdot\grad$, where $\vec{b}$ is divergence-free, i.e., $\div{\vec{b}} = 0$, and has vanishing normal component on the boundary, i.e., $\vec{b}\cdot\vec{n} = 0$.

\paragraph{Equivalent differential inclusion.}
Let us consider a heat equation on domain $\Omega\subset\IR^d$ over the time interval $[0,T]$ with initial and homogeneous Neumann boundary conditions as follows:
\begin{subequations}\label{eq:heat}
\begin{align}
\partial_t u - \laplace u &= 0 && \text{in}~\Omega\times(0,T],\\
u &= u^0 && \text{in}~\Omega\times\{0\}, \\
\grad{u}\cdot\normal &= 0 && \text{on}~\partial\Omega\times(0,T]. \label{eq:heat_BC}
\end{align}
\end{subequations}
Here, $\normal$ is the unit outer normal of domain $\Omega$.
The maximum principle \cite{evans2022partial} ensures $u^\mathrm{exact}(t)\in[m,M]$ for all $t\in[0,T]$, provided $u^0(\vec{x})\in[m,M]$ for all $\vec{x}\in\Omega$. Together with mass conservation, this gives $u^\mathrm{exact}(t)\in A$ for all $t\in[0,T]$, where $G=[m,M]$.
The differential inclusion \eqref{eq:abstract_pde_modified} specializes to
\begin{align}\label{eq:heat_modified}
\partial_t u - \laplace u + \partial\iota_{A}(u) \ni 0,
\end{align}
subject to the same initial and boundary conditions as \eqref{eq:heat}.
Since the heat equation is well-posed, to show the equivalence between \eqref{eq:heat} and \eqref{eq:heat_modified}, we only need to show the solution of \eqref{eq:heat_modified} is unique.
Notice $u^\mathrm{exact}$ is a solution of \eqref{eq:heat}, so it solves \eqref{eq:heat_modified}. 
The uniqueness follows from the energy argument given below for the convection-diffusion equation with $D=1$ and $\vec{b} = \vec{0}$.
\par
For the convection-diffusion equation, let us consider the following PDE model with initial and no-flux boundary conditions:
\begin{subequations}\label{eq:conv_diff}
\begin{align}
\partial_t u - D\laplace u + \vec{b}\cdot\grad u &= 0 && \text{in}~\Omega\times(0,T],\\
u &= u^0 && \text{in}~\Omega\times\{0\},\\
(-D\grad{u} + \vec{b}u)\cdot\normal &= 0 && \text{on}~\partial\Omega\times(0,T]. \label{eq:conv_diff_BC}
\end{align}
\end{subequations}
Here, the coefficient $D > 0$ is a constant and $\vec{b}\in W^{1,\infty}(\Omega)^d$ satisfies the divergence-free condition $\div{\vec{b}} = 0$ and $\vec{b}\cdot\normal = 0$ on $\partial\Omega$.
The maximum principle \cite{evans2022partial} gives: 
since $\vec{b}\cdot\normal=0$ on $\partial\Omega$ as assumed, if $u^0(\vec{x})\in[m,M]$ for all $\vec{x}\in\Omega$, then $u^\mathrm{exact}(t)\in[m,M]$ for all $t\in[0,T]$ and we take $G=[m,M]$. 
Thus, together with mass conservation, $u^\mathrm{exact}(t)\in A$ for all $t\in[0,T]$.
The differential inclusion \eqref{eq:abstract_pde_modified} specializes to
\begin{align}\label{eq:conv_diff_modified}
\partial_t{u} - D\laplace{u} + \vec{b}\cdot\grad u + \partial\iota_{A}(u) \ni 0,
\end{align}
subject to the same initial and boundary conditions as \eqref{eq:conv_diff}. Since \eqref{eq:conv_diff} is well-posed, to show the equivalence between \eqref{eq:conv_diff} and \eqref{eq:conv_diff_modified}, we only need to show \eqref{eq:conv_diff_modified} has a unique solution.
Since $u^\mathrm{exact}$ is the solution of \eqref{eq:conv_diff}, the maximum principle gives $u^\mathrm{exact}(t)\in A$, hence $0\in\partial\iota_{A}(u^\mathrm{exact}(t))$, and $u^\mathrm{exact}$ satisfies \eqref{eq:conv_diff_modified} with the subgradient selection equal to zero.
Let $u$ denote a solution of \eqref{eq:conv_diff_modified}. Set $w = u - u^\mathrm{exact}$. Then, there exists $\xi\in\partial\iota_{A}(u)$, such that
\begin{align}\label{eq:conv_diff_diff}
\partial_t w - D\laplace w + \vec{b}\cdot\grad w + \xi = 0.
\end{align}
Since both $u$ and $u^\mathrm{exact}$ satisfy the no-flux boundary condition, so does $w$. 
Let $\norm{\cdot}{L^2(O)}$ denote the $L^2$ norm over set $O$. Throughout this paper, when the underlying spatial domain is $\Omega$, we omit it from the notation of norms and seminorms, for instance we write $\norm{\cdot}{L^2}$ and $\abs{\cdot}_{H^{k+1}}$.
Taking the inner product of \eqref{eq:conv_diff_diff} with $w$, integrating by parts, and substituting $D\nabla w\cdot\normal=(\vec{b}\cdot\normal)w$ from the no-flux condition, we get
\begin{align*}
\frac{1}{2}\frac{\dd}{\dd t}\norm{w}{L^2}^2 + D\norm{\grad w}{L^2}^2 - \frac{1}{2}\int_{\partial\Omega}(\vec{b}\cdot\normal)w^2 + (\xi, w) = 0.
\end{align*}
From $\xi\in\partial\iota_{A}(u)$ and $u^\mathrm{exact}\in A$, by \eqref{eq:def_subgradient} we have $(\xi, u^\mathrm{exact} - u)\leq 0$, namely $(\xi,w)\geq 0$. Since $\vec{b}\cdot\normal=0$ holds on $\partial\Omega$, the boundary term vanishes, so $\frac{\dd}{\dd t}\norm{w}{L^2}^2\leq 0$ and $w(0)=0$ implies $w=0$.

\subsection{Operator splitting and numerical scheme}\label{sec:operator_splitting}
To construct an efficient numerical solver, let us split the differential inclusion \eqref{eq:abstract_pde_modified} into an unconstrained PDE and a subgradient flow of the indicator function. We have two subproblems
\begin{subequations}\label{eq:subproblem_HL}
\begin{align}
\mathcal{H}&: \quad \partial_t u + \mathcal{A}{u} = 0, \\
\mathcal{L}&: \quad \partial_t u + \partial\iota_{A}(u) \ni 0. \label{eq:subproblem_L}
\end{align}
\end{subequations}
We partition the time interval $(0,T]$ uniformly and let $\tau$ denote the time step size. 
Given solution $U^n$ at time step $t^n$ ($n\geq 0$), the Lie splitting advances the solution to time step $t^{n+1}$ by solving $\mathcal{H}$ and $\mathcal{L}$ sequentially. To be more specific, we have $U^{n+1} = \mathcal{L}(\tau)\circ\mathcal{H}(\tau) U^n$, namely we solve $\mathcal{H}$ on $[t^n, t^{n+1}]$ with initial $U^n$ to obtain an intermediate state $u^{n+1}$, then solve $\mathcal{L}$ on $[t^n, t^{n+1}]$ with initial $u^{n+1}$ to obtain $U^{n+1}$.
The step $\mathcal{L}$ projects $u^{n+1}$ onto $A$, yielding $U^{n+1}\in A$.

\paragraph{Subproblem $\mathcal{H}$.}
Let $\setE_h = \{E_i\}$ be a family of conforming nondegenerate (regular) quasi-uniform meshes of the domain $\Omega$ with maximum element diameter $h$. Let $\Gammah$ be the set of interior faces. For each interior face $e \in \Gammah$ shared by elements $E_{i^-}$ and $E_{i^+}$, with $i^- < i^+$, we define a unit normal vector $\normal_e$ that points from $E_{i^-}$ into $E_{i^+}$. For a boundary face $e \subset \partial\Omega$, we take $\normal_e$ to be the unit outer normal to $\partial\Omega$.
Let $\IQ_k(E_i)$ denote the space of tensor product polynomials of degree at most $k$ on an element $E_i$.
Define the following broken polynomial spaces, for any $k \geq 1$,
\begin{align*}
V_h &= \{v_h \in L^2(\Omega):~ \on{v_h}{E_i}\in\IQ_k(E_i),~ \forall E_i\in\setE_h\} \\
V_h^0 &= \{v_h\in V_h:~ (v_h, 1) = 0\}.
\end{align*}
For meshes with triangle and tetrahedron, the space $\IP_k(E_i)$, i.e. the set of all polynomials of degree at most $k$ on an element $E_i$, is used instead of $\IQ_k(E_i)$ in above definition.
The average and jump for $v$ on a boundary face coincide with its trace; and on interior faces they are defined by
\begin{equation*}
\avg{v}|_e = \frac{1}{2}\on{v}{E_{i^-}} + \frac{1}{2}\on{v}{E_{i^+}}, \quad
\jump{v}|_e = \on{v}{E_{i^-}} - \on{v}{E_{i^+}}, \quad 
\forall e = \partial E_{i^{-}}\cap\partial E_{i^{+}}.
\end{equation*}
We utilize the symmetric interior penalty DG (SIPG) for spatial discretization. For the heat equation with boundary condition \eqref{eq:heat_BC}, the bilinear form of $-\laplace{u}$ is defined by 
\begin{align*}
a_\mathcal{D}(u, v) &= \sum_{E \in \setE_h} \int_E \grad u \cdot \grad v - \sum_{e \in \Gammah} \int_e \avg{\grad{u}\cdot\normal_e}\jump{v} - \sum_{e \in \Gammah} \int_e\avg{\grad{v}\cdot\normal_e}\jump{u} + \frac{\sigma}{h}\sum_{e \in \Gammah} \int_e\jump{u}\jump{v},
\end{align*}
where $\sigma > 0$ is a penalty parameter. Define the energy semi-norm on $V_h$, which is a norm on space $V_h^0$, as follows
\begin{align*}
\norm{v}{\mathrm{DG}}^2 = \sum_{E \in \setE_h} \norm{\grad v}{L^2(E)}^2 + \frac{\sigma}{h}\sum_{e \in \Gammah} \norm{\jump{v}}{L^2(e)}^2.
\end{align*}
The implicit Euler time stepping for the heat equation reads: given $U_h^n\in V_h$, find $u_h^{n+1}\in V_h$ such that
\begin{align*}
(u_h^{n+1} - U_h^n, v_h) + \tau a_\mathcal{D}(u_h^{n+1}, v_h) = 0, \quad \forall v_h \in V_h.
\end{align*}
Recall the prescribed divergence-free vector field $\vec{b}\in W^{1,\infty}(\Omega)^d$. The bilinear form of the advection term $\vec{b}\cdot\grad{u}$ discretized with upwind flux is defined by
\begin{align*}
a_{\mathrm{adv}}(u, v) = -\sum_{E\in\setE_h}\int_E u \vec{b}\cdot\grad v + \sum_{e\in\Gammah}\int_e \upwind{u}(\vec{b}\cdot\normal_e)\jump{v},
\end{align*}
where the upwind of $u$ on an interior face $e = E_{i^-}\cap E_{i^+}$ with normal $\normal_e = \on{\normal}{E_{i^-}}$ is defined by
\begin{align*}
\upwind{u} = \begin{cases}
\on{u}{E_{i^-}} & \text{if}~ \vec{b}\cdot\normal_e \geq 0, \\
\on{u}{E_{i^+}} & \text{if}~ \vec{b}\cdot\normal_e < 0.
\end{cases}
\end{align*}
Note that with the no-flux boundary condition \eqref{eq:conv_diff_BC}, integrating by parts shows that the boundary face contributions from $-D\laplace{u} + \vec{b}\cdot\grad{u}$ cancel. Hence, the sums in $a_\mathcal{D}$ and $a_{\mathrm{adv}}$ are taken over interior faces $\Gammah$ only.
The implicit Euler time discretization for the convection-diffusion equation \eqref{eq:conv_diff} reads: given $U_h^n\in V_h$, find $u_h^{n+1}\in V_h$ such that
\begin{align*}
(u_h^{n+1} - U_h^n, v_h) + \tau D a_\mathcal{D}(u_h^{n+1}, v_h) + \tau a_{\mathrm{adv}}(u_h^{n+1}, v_h) = 0, \quad \forall v_h \in V_h.
\end{align*}
The intermediate state $u_h^{n+1} \in V_h$ does not in general satisfy $u_h^{n+1} \in A$. We therefore solve the subproblem $\mathcal{L}$, which projects $u_h^{n+1}$ onto set $A$.

\paragraph{Subproblem $\mathcal{L}$.}
To efficiently solve the subproblem $\mathcal{L}$, we apply the implicit Euler method with a finite volume spatial discretization.
Let $\IP_0(\setE_h)$ denote the set of piecewise constant functions on $\setE_h$. We approximate the continuous constraint set $A$ in \eqref{eq:set_A} by the associated discrete piecewise constant analogue $\Lambda = \Lambda_1\cap \Lambda_2$, where
\begin{align*}
\Lambda_1 &= \{v_h \in \IP_0(\setE_h):~ (v_h,1) = (u^0,1)\}\\
\text{and}\quad
\Lambda_2 &= \{v_h \in \IP_0(\setE_h):~ m \leq \on{v_h}{E} \leq M,~ \forall E \in \setE_h \}.
\end{align*}
Note that $\Lambda_1$ represents the conservation constraint and $\Lambda_2$ represents the bound-preserving constraint on cell averages.
After obtaining the discrete solution in $\IP_0(\setE_h)$, we recover high-order spatial accuracy by adjusting the cell averages.
For a given function $v$, its cell average $\overline{v}$ is piecewise constant, such that on each element $E \in \setE_h$,  
\begin{align*}
\on{\overline{v}}{E} = \frac{1}{\abs{E}}\int_E v.
\end{align*}
Applying implicit Euler time stepping to \eqref{eq:subproblem_L}, the fully discrete scheme reads: given $u_h^{n+1} \in V_h$, compute $\overline{U_h^{n+1}} \in \IP_0(\setE_h)$, such that
\begin{equation}\label{eq:resolvent_eq}
\frac{1}{\tau}(\overline{U_h^{n+1}} - \overline{u_h^{n+1}}) + \xi_h^{n+1} = 0, ~\text{where}~ 
\xi_h^{n+1} \in \partial\iota_{\Lambda}(\overline{U_h^{n+1}})
\quad\Rightarrow\quad
\overline{U_h^{n+1}} + \tau \partial\iota_{\Lambda}(\overline{U_h^{n+1}}) \ni \overline{u_h^{n+1}}.
\end{equation}
Since $\Lambda$ is nonempty, closed, and convex, the solution of \eqref{eq:resolvent_eq} is the resolvent of the maximal monotone operator $\partial\iota_{\Lambda}$ applied to $\overline{u_h^{n+1}}$. We have
\begin{align}\label{eq:cell_avg_limiter_model0}
\overline{U_h^{n+1}} 
&= (I + \tau \partial\iota_{\Lambda})^{-1}(\overline{u_h^{n+1}}) 
= \prox^\tau_{\iota_{\Lambda}} (\overline{u_h^{n+1}}) \nonumber\\
&= \operatorname*{arg\,min}_{x_h \in \IP_0(\setE_h)}
\Big(\frac{1}{2\tau}\norm{x_h - \overline{u_h^{n+1}}}{L^2}^2 + \iota_{\Lambda}(x_h)\Big)\\
&= \operatorname*{arg\,min}_{x_h \in \Lambda} \frac{1}{2}\norm{x_h - \overline{u_h^{n+1}}}{L^2}^2.\nonumber
\end{align}
The objective in \eqref{eq:cell_avg_limiter_model0} is strongly convex. Thus, the minimizer exists and is unique.
Finally, we reconstruct the high-order solution in space by adjusting the cell average of $u_h^{n+1}$, as follows
\begin{align}\label{eq:adjust_cell_avg}
U_h^{n+1} = u_h^{n+1} - \overline{u_h^{n+1}} + \overline{U_h^{n+1}}.
\end{align}  
The formulations defined in \eqref{eq:cell_avg_limiter_model0} and \eqref{eq:adjust_cell_avg} form the optimization-based cell average limiter, which guarantees that the cell averages stay in the admissible set. 
Note that in \eqref{eq:resolvent_eq}, \eqref{eq:cell_avg_limiter_model0}, and the right-hand side of \eqref{eq:adjust_cell_avg}, the symbol $\overline{U_h^{n+1}}$ is used as a whole to denote the projection of $\overline{u_h^{n+1}}$ onto $\Lambda$, rather than the cell average operator applied to $U_h^{n+1}$, which has not yet been defined at that point. This is the only place where the overline appears as part of a symbol rather than as an operator. Elsewhere, overlines denote taking cell averages.
\begin{theorem}\label{thm:mass_conservation}
The PDE models \eqref{eq:heat} and \eqref{eq:conv_diff} and associated numerical schemes enjoy global conservation. In addition, the cell averages of the exact solutions stay in set $\Lambda$.
\end{theorem}
\begin{proof}
The $L^2$ projection for initial data is conservative, namely $(u_h^0,1) = (u^0,1)$.
The DG scheme is conservative, namely $(u_h^{n+1},1) = (U_h^{n},1)$; the optimization problem uses conservation as a constraint and \eqref{eq:adjust_cell_avg} does not change the global mass, namely $(U_h^{n+1},1) = (u_h^{n+1},1)$.
Putting these together, induction gives $(u_h^n,1) = (U_h^n,1) = (u^0,1)$ for any $n\geq0$. 
\par
The PDEs are conservative, namely the solutions satisfy $(u^{\mathrm{exact}}(t),1) = (u^0,1)$. Therefore, the cell averages of the exact solutions belong to set $\Lambda_1$, namely we have $(\overline{u^\mathrm{exact}}, 1) = (u^0,1)$.
Since the maximum principle gives $u^\mathrm{exact}(\vec{x},t)\in[m,M]$ for all $\vec{x}\in\Omega$ and $t\in[0,T]$, the cell averages also satisfy $m\leq\overline{u^\mathrm{exact}}\leq M$, so $\overline{u^\mathrm{exact}}\in\Lambda$.
\end{proof}
\par
After each time step, the cell averages of $U_h^{n+1} = \mathcal{L}(\tau)\circ\mathcal{H}(\tau)U_h^n$ are guaranteed to lie in the admissible set. The high-order moments are subsequently modified by the classical Zhang--Shu scaling limiter, which preserves the local mass conservation \cite{zhang2010maximum}.
We emphasize again that the analysis of this paper focuses exclusively on the cell average limiter, i.e., on solving \eqref{eq:subproblem_HL} without Zhang--Shu scaling.
\begin{remark}
One may consider Strang splitting $U^{n+1} = \mathcal{L}(\frac{\tau}{2})\circ\mathcal{H}(\tau)\circ\mathcal{L}(\frac{\tau}{2}) U^n$, namely we solve $\mathcal{L}$ on $[t^n, t^n + \frac{\tau}{2}]$ with initial $U^n$ to obtain $u^{n+1/2}$, then solve $\mathcal{H}$ on $[t^n, t^{n+1}]$ with initial $u^{n+1/2}$ to obtain $u^{n+1}$, finally solve $\mathcal{L}$ on $[t^n + \frac{\tau}{2}, t^{n+1}]$ with initial $u^{n+1}$ to obtain $U^{n+1}$.
For any $\tau > 0$, the proximal operator $\prox^\tau_{\iota_{\Lambda}}$ projects onto $\Lambda$. Thus, the $\mathcal{L}$ step is independent of time step size.
In particular, since $\overline{U_h^n}\in\Lambda$ at the start of each time step, the opening half-step $\mathcal{L}(\frac{\tau}{2})$ leaves $U_h^n$ unchanged, giving $\mathcal{L}(\frac{\tau}{2}) U_h^n = U_h^n$. Hence, the Strang splitting reduces to $\mathcal{L}(\frac{\tau}{2})\circ\mathcal{H}(\tau)$, which coincides with the Lie splitting because the $\mathcal{L}$ step does not depend on time step size.
\end{remark}

\section{Properties of the cell average limiter}\label{sec:cell_avg_limiter}
An optimization-based cell average limiter finds a piecewise constant polynomial that minimizes the $L^2$ distance to the cell average of the DG solution, subject to the conservation and bound-preserving constraints.
In this section, we analyze the properties of the cell average limiter at a single time step. For convenience, we suppress the time superscript throughout this section.

\paragraph{Projection onto admissible set.}
Recalling \eqref{eq:cell_avg_limiter_model0}, given a DG solution $u_h \in V_h$, the cell average limiter seeks $x_h \in \IP_0(\setE_h)$ by solving
\begin{align}\label{eq:opt_model2}
\min_{x_h\in \IP_0(\setE_h)} \norm{x_h - \overline{u_h}}{L^2}^2 + \iota_{\Lambda_1}(x_h) + \iota_{\Lambda_2}(x_h).
\end{align}
By definition of the indicator function, we have the following constrained minimization problem:
\begin{align}\label{eq:opt_model1}
\min_{x_h\in \IP_0(\setE_h)} \norm{x_h - \overline{u_h}}{L^2}^2 
\quad\text{subject to}\quad
\int_\Omega x_h = \int_\Omega u_h
\quad\text{and}\quad
m \leq x_h \leq M.
\end{align}
Since $\Lambda_1$ and $\Lambda_2$ are both convex and closed, the indicators $\iota_{\Lambda_1}$ and $\iota_{\Lambda_2}$ are convex functions. Thus, \eqref{eq:opt_model2} is a strongly convex minimization problem. The unique minimizer is
\begin{align}\label{eq:opt_model3}
x_h^\ast = \mathcal{Q}\overline{u_h} = \operatorname*{arg\,min}_{x_h \in \Lambda} \norm{x_h - \overline{u_h}}{L^2}^2.
\end{align}
Consistent with \eqref{eq:adjust_cell_avg}, the post-processed DG solution is defined by $u_h^\mathrm{lim} = u_h - \overline{u_h} + x_h^\ast$.
We begin with classical nonexpansiveness properties of metric projections onto closed convex sets in Hilbert spaces.
\begin{lemma}\label{lem:proj_nonexpansive}
Let $H$ denote a Hilbert space associated with norm $\norm{\cdot}{}$. Let $\Lambda\subset H$ be a closed convex set. For any given $u\in H$, define projection operator $\mathcal{Q}: H \rightarrow \Lambda$ as follows: 
\begin{align*}
\mathcal{Q}u = \operatorname*{arg\,min}_{z\in \Lambda} \norm{z - u}{}^2.
\end{align*}
Then, the operator $\mathcal{Q}$ is nonexpansive, namely for any $u,w\in H$, we have $\norm{\mathcal{Q}u - \mathcal{Q}w}{}\leq\norm{u-w}{}$.
\end{lemma}
\begin{proof}
Since $\mathcal{Q}u\in \Lambda$ minimizes $\norm{z-u}{}^2$ over the convex closed set $\Lambda$, $\forall z \in \Lambda$ and $\forall \eta \in [0,1]$, the convex combination $(1-\eta)\mathcal{Q}u + \eta z \in \Lambda$. Define
\begin{align*}
\phi(\eta) = \norm{u - ((1-\eta)\mathcal{Q}u + \eta z)}{}^2
= \eta^2\norm{z - \mathcal{Q}u}{}^2 - 2\eta(u - \mathcal{Q}u, z - \mathcal{Q}u) + \norm{u - \mathcal{Q}u}{}^2.
\end{align*}
If $z = \mathcal{Q}u$, then the inner product $(u - \mathcal{Q}u, z - \mathcal{Q}u) \leq 0$ holds. 
Otherwise, $\phi$ is a convex quadratic function with respect to $\eta$ and reaches its minimum over $[0,1]$ at $\eta=0$. Thus, its axis of symmetry lies at $\eta^\ast\leq0$, which yields
\begin{align}\label{eq:proj_nonexpansive_1}
\eta^\ast = \frac{(u - \mathcal{Q}u, z - \mathcal{Q}u)}{\norm{z - \mathcal{Q}u}{}^2} \leq 0
\quad\Rightarrow\quad
(u - \mathcal{Q}u, z - \mathcal{Q}u) \leq 0, \quad \forall z \in \Lambda.
\end{align}
For any $u,w\in H$, noting that $\mathcal{Q}u\in \Lambda$ and $\mathcal{Q}w\in \Lambda$, applying the inequality above, we have
\begin{align*}
(u - \mathcal{Q}u, \mathcal{Q}w - \mathcal{Q}u) \leq 0 
\quad\text{and}\quad 
(w - \mathcal{Q}w, \mathcal{Q}u - \mathcal{Q}w) \leq 0.
\end{align*}
If $\mathcal{Q}u = \mathcal{Q}w$, the conclusion is trivial. Otherwise, adding the above two inequalities and rearranging, by Cauchy--Schwarz's inequality, we have
\begin{align*}
\norm{\mathcal{Q}u - \mathcal{Q}w}{}^2 
\leq (u - w, \mathcal{Q}u - \mathcal{Q}w) 
\leq \norm{u - w}{}\norm{\mathcal{Q}u - \mathcal{Q}w}{}.
\end{align*}
Therefore, we conclude the proof.
\end{proof}
\begin{lemma}\label{lem:proj_firm}
Let space $H$, set $\Lambda$, and projection operator $\mathcal{Q}$ be as in Lemma~\ref{lem:proj_nonexpansive}. For any $u\in H$ and any $z\in \Lambda$, we have $\norm{u - \mathcal{Q}u}{}^2 + \norm{\mathcal{Q}u - z}{}^2 \leq \norm{u - z}{}^2$.
\end{lemma}
\begin{proof}
Recall the proof of Lemma~\ref{lem:proj_nonexpansive}. The \eqref{eq:proj_nonexpansive_1} gives $(u - \mathcal{Q}u, z - \mathcal{Q}u)\leq0$ for all $z\in\Lambda$. 
Thus, from $u - z = (u - \mathcal{Q}u) + (\mathcal{Q}u - z)$, we obtain
\begin{align*}
\norm{u - z}{}^2
= \norm{u - \mathcal{Q}u}{}^2 + 2(u - \mathcal{Q}u, \mathcal{Q}u - z) + \norm{\mathcal{Q}u - z}{}^2
\geq \norm{u - \mathcal{Q}u}{}^2 + \norm{\mathcal{Q}u - z}{}^2.
\end{align*}
Therefore, we conclude the proof.
\end{proof}

\paragraph{Accuracy preservation and stability.}
Applying Lemma~\ref{lem:proj_nonexpansive}, we establish the optimality of the cell average limiter \eqref{eq:opt_model1} as follows.
\begin{theorem}\label{thm:accuracy_L2}
The modified cell averages $x_h^\ast$ are not worse than the original DG cell averages $\overline{u_h}$ as an approximation to the cell average of the exact solution $\overline{u^\mathrm{exact}}$, measured in the $L^2$ norm, namely the inequality $\norm{x_h^\ast - \overline{u^\mathrm{exact}}}{L^2} \leq \norm{\overline{u_h} - \overline{u^\mathrm{exact}}}{L^2}$ holds. 
\end{theorem} 
\begin{proof}
The optimization model \eqref{eq:opt_model1} defines the projection operator $\mathcal{Q}$ onto the closed convex set $\Lambda = \Lambda_1 \cap \Lambda_2 \subset \IP_0(\setE_h)$. By definition, $x_h^\ast = \mathcal{Q}\overline{u_h}$ is the projection of $\overline{u_h}$ onto $\Lambda$. From Theorem~\ref{thm:mass_conservation}, we know $\overline{u^\mathrm{exact}}\in\Lambda$. We have $\mathcal{Q}\overline{u^\mathrm{exact}} = \overline{u^\mathrm{exact}}$.
Applying Lemma~\ref{lem:proj_nonexpansive}, we get 
\begin{align*}
\norm{x_h^\ast - \overline{u^\mathrm{exact}}}{L^2} 
= \norm{\mathcal{Q}\overline{u_h} - \mathcal{Q}\overline{u^\mathrm{exact}}}{L^2}
\leq\norm{\overline{u_h} - \overline{u^\mathrm{exact}}}{L^2}.
\end{align*}
Therefore, we conclude the proof.
\end{proof}
For any $u\in L^2(\Omega)$, we can decompose its $L^2$ norm into the sum of the $L^2$ norm of cell averages and high-order moments (Pythagorean identity), namely by $(\overline{u},u-\overline{u}) = 0$, we have
\begin{align}\label{eq:cell_avg_decomposition}
\norm{u}{L^2}^2 
= \norm{\overline{u}}{L^2}^2 + 2(\overline{u},u-\overline{u}) + \norm{u-\overline{u}}{L^2}^2 
= \norm{\overline{u}}{L^2}^2 + \norm{u-\overline{u}}{L^2}^2.
\end{align}
\begin{theorem}\label{thm:limiter_accuracy}
Suppose the exact solution $u^\mathrm{exact}\in [m,M]$. The cell average limiter \eqref{eq:opt_model1} does not worsen the $L^2$ approximation error to the exact solution. We have
\begin{align*}
\norm{u_h^\mathrm{lim} - u^\mathrm{exact}}{L^2} \leq \norm{u_h - u^\mathrm{exact}}{L^2}.
\end{align*}
\end{theorem}
\begin{proof}
Recall $u_h^\mathrm{lim} = u_h - \overline{u_h} + x_h^\ast$. As $x_h^\ast$ is piecewise constant, on each cell $E_i\in\setE_h$, we have
\begin{align*}
\on{\overline{u_h^\mathrm{lim}}}{E_i} 
= \frac{1}{\abs{E_i}} \int_{E_i} (u_h - \overline{u_h} + x_h^\ast)
= \frac{1}{\abs{E_i}} \int_{E_i} u_h - \on{\overline{u_h}}{E_i}\frac{1}{\abs{E_i}}\int_{E_i} 1 + \on{x_h^\ast}{E_i}\frac{1}{\abs{E_i}}\int_{E_i} 1 
= \on{x_h^\ast}{E_i}.
\end{align*} 
Thus, on each cell $E_i\in\setE_h$, by definition of $u_h^\mathrm{lim}$, we obtain
\begin{align*}
\on{(u_h^\mathrm{lim}-\overline{u_h^\mathrm{lim}})}{E_i}
= \on{(u_h^\mathrm{lim} - x_h^\ast)}{E_i}
= \on{(u_h - \overline{u_h})}{E_i}.
\end{align*}
Applying decomposition \eqref{eq:cell_avg_decomposition} to $u_h^\mathrm{lim} - u^\mathrm{exact}$, because of the linearity of the cell average, we get
\begin{align*}
\norm{u_h^\mathrm{lim} - u^\mathrm{exact}}{L^2}^2
&= \norm{\overline{u_h^\mathrm{lim} - u^\mathrm{exact}}}{L^2}^2
+ \norm{(u_h^\mathrm{lim} - u^\mathrm{exact}) - \overline{u_h^\mathrm{lim} - u^\mathrm{exact}}}{L^2}^2\\
&= \norm{\overline{u_h^\mathrm{lim}} - \overline{u^\mathrm{exact}}}{L^2}^2
+ \norm{(u_h^\mathrm{lim} - \overline{u_h^\mathrm{lim}}) - (u^\mathrm{exact} - \overline{u^\mathrm{exact}})}{L^2}^2\\
&= \norm{x_h^\ast - \overline{u^\mathrm{exact}}}{L^2}^2
+ \norm{(u_h - \overline{u_h}) - (u^\mathrm{exact} - \overline{u^\mathrm{exact}})}{L^2}^2.
\end{align*}
By Theorem~\ref{thm:accuracy_L2}, we have $\norm{x_h^\ast - \overline{u^\mathrm{exact}}}{L^2} \leq \norm{\overline{u_h} - \overline{u^\mathrm{exact}}}{L^2}$.
Applying decomposition \eqref{eq:cell_avg_decomposition} again to $u_h - u^\mathrm{exact}$, we obtain
\begin{align*}
\norm{u_h^\mathrm{lim} - u^\mathrm{exact}}{L^2}^2
&\leq \norm{\overline{u_h} - \overline{u^\mathrm{exact}}}{L^2}^2
+ \norm{(u_h - \overline{u_h}) - (u^\mathrm{exact} - \overline{u^\mathrm{exact}})}{L^2}^2\\
&= \norm{\overline{u_h - u^\mathrm{exact}}}{L^2}^2
+ \norm{(u_h - u^\mathrm{exact}) - \overline{u_h - u^\mathrm{exact}}}{L^2}^2
= \norm{u_h - u^\mathrm{exact}}{L^2}^2.
\end{align*}
Therefore, we conclude the proof.
\end{proof}
\begin{remark}
For a fixed time step, Theorem~\ref{thm:limiter_accuracy} implies that the cell average limiter preserves the high-order spatial accuracy of the DG solution. If $u_h$ satisfies $\norm{u_h - u^\mathrm{exact}}{L^2} = \mathcal{O}(h^{k+1})$, then $\norm{u_h^\mathrm{lim} - u^\mathrm{exact}}{L^2} = \mathcal{O}(h^{k+1})$ as well.
\end{remark}
For stability of the limiter, we have the following result that the cell average limiter does not increase the $L^2$ norm.
\begin{theorem}
Assume the exact solution satisfies $u^\mathrm{exact}\in[m, M]$ on $\Omega$ and the numerical algorithm is conservative. The $L^2$ norm of the limited solution and its cell averages satisfy
\begin{align}
\norm{x_h^\ast}{L^2} &\leq \norm{\overline{u_h}}{L^2}, \label{eq:stability_cell_avg} \\
\norm{u_h^\mathrm{lim}}{L^2} &\leq \norm{u_h}{L^2}. \label{eq:stability_limiter}
\end{align}
\end{theorem}
\begin{proof}
Since the exact solution satisfies $u^\mathrm{exact}\in[m,M]$ and the algorithm is conservative, from the total mass
\begin{align*}
b = \int_\Omega u_h = \int_\Omega u^\mathrm{exact}\quad\text{and}\quad m \leq u^\mathrm{exact} \leq M,
\end{align*} 
we get $m\abs{\Omega} \leq b \leq M\abs{\Omega}$. Hence $m \leq b/\abs{\Omega} \leq M$.
Define a piecewise constant function $\xi_h$ by $\xi_h|_{E_i} = b/\abs{\Omega}$ for all $E_i \in \setE_h$. By construction, $\xi_h$ satisfies the conservation constraint $\int_\Omega \xi_h = b$ and the bound constraint $m \leq \xi_h \leq M$, so $\xi_h \in \Lambda$.
Notice that the algorithm is conservative, $\int_\Omega \overline{u_h} = \int_\Omega u_h = b = \int_\Omega \xi_h$, and $x_h^\ast = \mathcal{Q}(\overline{u_h}) \in \Lambda$ gives $\int_\Omega x_h^\ast = b$ as well.
Thus, both $\overline{u_h} - \xi_h$ and $x_h^\ast - \xi_h$ have vanishing integral over $\Omega$.
Since $\xi_h$ is a constant function on $\Omega$, it is $L^2$ orthogonal to $\overline{u_h} - \xi_h$ and $x_h^\ast - \xi_h$. 
We have
\begin{align}\label{eq:stability_1}
\norm{\overline{u_h}}{L^2}^2 = \norm{\xi_h}{L^2}^2 + \norm{\overline{u_h} - \xi_h}{L^2}^2
\quad\text{and}\quad
\norm{x_h^\ast}{L^2}^2 = \norm{\xi_h}{L^2}^2 + \norm{x_h^\ast - \xi_h}{L^2}^2.
\end{align}
From $\xi_h \in \Lambda$, we have $\mathcal{Q}(\xi_h) = \xi_h$. Taking $w = \xi_h$ in Lemma~\ref{lem:proj_nonexpansive},
\begin{align}\label{eq:stability_2}
\norm{x_h^\ast - \xi_h}{L^2} = \norm{\mathcal{Q}(\overline{u_h}) - \mathcal{Q}(\xi_h)}{L^2} \leq \norm{\overline{u_h} - \xi_h}{L^2}.
\end{align}
Combining \eqref{eq:stability_1} and \eqref{eq:stability_2}, we get \eqref{eq:stability_cell_avg}.
Recall the post-processed solution $u_h^\mathrm{lim} = u_h - \overline{u_h} + x_h^\ast$. Then, on each cell $E\in\setE_h$, we have
\begin{align*}
\on{\overline{u_h^\mathrm{lim}}}{E} 
= \frac{1}{\abs{E}} \int_E (u_h - \overline{u_h} + x_h^\ast)
= \frac{1}{\abs{E}} \int_E u_h - \on{\overline{u_h}}{E}\frac{1}{\abs{E}}\int_E 1 + \on{x_h^\ast}{E}\frac{1}{\abs{E}}\int_E 1 
= \on{x_h^\ast}{E}.
\end{align*} 
Thus, we obtain
\begin{align*}
\norm{\overline{u_h^\mathrm{lim}}}{L^2}^2
&= \sum_{E\in\setE_h} \int_E \abs{\overline{u_h^\mathrm{lim}}}^2 
= \sum_{E\in\setE_h} \int_E \abs{x_h^\ast}^2
= \norm{x_h^\ast}{L^2}^2, \\
\norm{u_h^\mathrm{lim}-\overline{u_h^\mathrm{lim}}}{L^2}^2
&= \sum_{E\in\setE_h} \int_E \abs{u_h^\mathrm{lim}-\overline{u_h^\mathrm{lim}}}^2
= \sum_{E\in\setE_h} \int_E \abs{u_h^\mathrm{lim}-x_h^\ast}^2
= \sum_{E\in\setE_h} \int_E \abs{u_h - \overline{u_h}}^2
= \norm{u_h - \overline{u_h}}{L^2}^2.
\end{align*}
Applying decomposition \eqref{eq:cell_avg_decomposition}, we get the following stability
\begin{align*}
\norm{u_h^\mathrm{lim}}{L^2}^2 
&= \norm{\overline{u_h^\mathrm{lim}}}{L^2}^2 + \norm{u_h^\mathrm{lim}-\overline{u_h^\mathrm{lim}}}{L^2}^2
= \norm{x_h^\ast}{L^2}^2 + \norm{u_h - \overline{u_h}}{L^2}^2 \\
&\leq \norm{\overline{u_h}}{L^2}^2 + \norm{u_h - \overline{u_h}}{L^2}^2
= \norm{u_h}{L^2}^2. 
\end{align*}
Therefore, we obtain \eqref{eq:stability_limiter}.
\end{proof}
The optimization-based cell average limiter never increases the oscillation in the means. For a piecewise constant polynomial $v_h \in \IP_0(\setE_h)$, to measure oscillation across cell interfaces, define the total variation
\begin{align}\label{eq:TV_def}
\mathrm{TV}(v_h) = \sum_{e\in\Gammah}\abs{e} \left|\jump{v_h}\right|.
\end{align}
\begin{theorem}\label{thm:tvdm}
The cell average limiter preserves the order of the cell averages, namely $\overline{u_h}|_{E_1} \geq \overline{u_h}|_{E_2}$ implies $x_h^\ast|_{E_1} \geq x_h^\ast|_{E_2}$ for all $E_1, E_2\in\setE_h$. Furthermore, it does not increase the total variation of cell averages. We have $\mathrm{TV}(x_h^\ast) \leq \mathrm{TV}(\overline{u_h})$.
\end{theorem}
\begin{proof}
Applying inequality \eqref{eq:proj_nonexpansive_1} with $x_h^\ast = \mathcal{Q}\overline{u_h}$, we get $(\overline{u_h} - x_h^\ast,~ z_h - x_h^\ast) \leq 0$ for any $z_h\in\Lambda$.
Let $E_1$ and $E_2$ be distinct cells with $x_h^\ast|_{E_1} < M$ and $x_h^\ast|_{E_2} > m$. For $\epsilon > 0$, define a piecewise constant polynomial
\begin{align*}
z_h(\vec{x}) = \begin{cases}
x_h^\ast(\vec{x}) + \frac{1}{\abs{E_1}}\epsilon, & \text{if}~ \vec{x}\in E_1, \\
x_h^\ast(\vec{x}) - \frac{1}{\abs{E_2}}\epsilon, & \text{if}~ \vec{x}\in E_2, \\
x_h^\ast(\vec{x}), & \text{if}~ \vec{x}\notin E_1\cup E_2.
\end{cases}
\end{align*}
Taking $\epsilon = \frac{1}{2}\min((M - x_h^\ast|_{E_1})\abs{E_1}, (x_h^\ast|_{E_2} - m)\abs{E_2}) > 0$, we have $m \leq z_h \leq M$ and $\int_\Omega z_h = \int_\Omega x_h^\ast$, namely $z_h \in \Lambda$. Therefore, we have
\begin{align*}
0 \geq (\overline{u_h} - x_h^\ast,~ z_h - x_h^\ast) = \epsilon(\overline{u_h} - x_h^\ast)|_{E_1} - \epsilon(\overline{u_h} - x_h^\ast)|_{E_2}
\quad\Rightarrow\quad
(\overline{u_h} - x_h^\ast)|_{E_1} \leq (\overline{u_h} - x_h^\ast)|_{E_2},
\end{align*}
which rearranges to
\begin{align}\label{eq:pair_order}
x_h^\ast|_{E_1} - x_h^\ast|_{E_2} \geq \overline{u_h}|_{E_1} - \overline{u_h}|_{E_2},
\quad \forall E_1 \neq E_2 ~\text{with}~ x_h^\ast|_{E_1} < M ~\text{and}~ x_h^\ast|_{E_2} > m.
\end{align}
Then $\overline{u_h}|_{E_1} \geq \overline{u_h}|_{E_2}$ implies $x_h^\ast|_{E_1} \geq x_h^\ast|_{E_2}$.
In the remaining cases $x_h^\ast|_{E_1} = M$ or $x_h^\ast|_{E_2} = m$, since $m \leq x_h^\ast \leq M$ on $\Omega$, the inequality $x_h^\ast|_{E_1} \geq x_h^\ast|_{E_2}$ holds trivially.
In all cases, the cell average limiter preserves the order of the cell averages.
\par
Let $e = E_1\cap E_2 \in\Gammah$, where $E_1$ and $E_2$ are labeled so that $\overline{u_h}|_{E_1} \leq \overline{u_h}|_{E_2}$. This labeling gives $\abs{\llbracket\overline{u_h}\rrbracket} = \overline{u_h}|_{E_2} - \overline{u_h}|_{E_1}$.
Since the order of the cell averages is preserved, we have $x_h^\ast|_{E_1} \leq x_h^\ast|_{E_2}$, so that $\abs{\llbracket{x_h^\ast}\rrbracket} = x_h^\ast|_{E_2} - x_h^\ast|_{E_1}$.
If $x_h^\ast|_{E_1} = x_h^\ast|_{E_2}$, then $\abs{\llbracket{x_h^\ast}\rrbracket} = 0 \leq \abs{\llbracket{\overline{u_h}}\rrbracket}$.
If $x_h^\ast|_{E_1} < x_h^\ast|_{E_2}$, then $m \leq x_h^\ast|_{E_1} < x_h^\ast|_{E_2} \leq M$, so $x_h^\ast|_{E_1} < M$ and $x_h^\ast|_{E_2} > m$. Using \eqref{eq:pair_order}, we have
\begin{align*}
\abs{\llbracket{x_h^\ast}\rrbracket} = x_h^\ast|_{E_2} - x_h^\ast|_{E_1} \leq \overline{u_h}|_{E_2} - \overline{u_h}|_{E_1} = \abs{\llbracket{\overline{u_h}}\rrbracket}.
\end{align*}
Multiplying by $\abs{e}$, summing over $e\in\Gammah$ and using \eqref{eq:TV_def}, we get $\mathrm{TV}(x_h^\ast) \leq \mathrm{TV}(\overline{u_h})$.
\end{proof}
\begin{remark}
In one dimension an interior face is a single node, where $\abs{e} = 1$ by convention. The definition \eqref{eq:TV_def} then reduces to the classical total variation of the sequence of cell averages.
Theorem~\ref{thm:tvdm} holds in any dimension, and it concerns the postprocessing step alone. One application of the cell average limiter does not increase \eqref{eq:TV_def}. This does not make the fully discrete scheme total variation diminishing, since the DG step itself need not be.
\end{remark}

\paragraph{Nonexpansiveness.}
Let $\mathcal{P}_h$ denote the composition of the optimization procedure $\mathcal{Q}$ defined in \eqref{eq:opt_model3} with the modification of cell averages \eqref{eq:adjust_cell_avg}, namely $\mathcal{P}_h v = v - \overline{v} + \mathcal{Q}\overline{v}$ for all $v \in L^2(\Omega)$.
In particular, the post-processed DG solution is $u_h^\mathrm{lim} = \mathcal{P}_h u_h$.
\begin{theorem}\label{thm:P_nonexpansive}
The cell average limiter $\mathcal{P}_h$ is nonexpansive in the $L^2$ norm: for any $v, w \in L^2(\Omega)$,
\begin{align*}
\norm{\mathcal{P}_h v - \mathcal{P}_h w}{L^2} \leq \norm{v - w}{L^2}.
\end{align*}
\end{theorem}
\begin{proof}
By definition of $\mathcal{P}_h$, we have
$\mathcal{P}_h v - \mathcal{P}_h w = ((v-w) - \overline{(v-w)}) + (\mathcal{Q}\overline{v} - \mathcal{Q}\overline{w})$.
The first term $(v-w) - \overline{(v-w)}$ has zero cell average on every element, while the second term $\mathcal{Q}\overline{v} - \mathcal{Q}\overline{w}$ is piecewise constant. Thus, these two terms are $L^2$ orthogonal. Applying Lemma~\ref{lem:proj_nonexpansive} to the second term and the decomposition \eqref{eq:cell_avg_decomposition} to $v - w$, we obtain
\begin{align*}
\norm{\mathcal{P}_h v - \mathcal{P}_h w}{L^2}^2
&= \norm{(v-w) - \overline{(v-w)}}{L^2}^2 + \norm{\mathcal{Q}\overline{v} - \mathcal{Q}\overline{w}}{L^2}^2 \\
&\leq \norm{(v-w) - \overline{(v-w)}}{L^2}^2 + \norm{\overline{v} - \overline{w}}{L^2}^2
= \norm{v-w}{L^2}^2.
\end{align*}
Therefore, we conclude the proof.
\end{proof}
In particular, taking $w = u^\mathrm{exact}$ recovers Theorem~\ref{thm:limiter_accuracy}, since Theorem~\ref{thm:mass_conservation} gives $\overline{u^\mathrm{exact}}\in\Lambda$ and hence $\mathcal{P}_h u^\mathrm{exact} = u^\mathrm{exact}$.
\begin{remark}\label{rem:steady_elliptic}
The analysis in this section is independent of the time-stepping scheme and holds at a single time step. In particular, the results apply directly to the time-independent problem: the cell average limiter is nonexpansive, and preserves both high-order accuracy and $L^2$ stability for a high-order DG discretization of a steady elliptic problem.
\end{remark}

\section{Error analysis}\label{sec:error_analysis}
In this section, we establish $L^2$ error estimates for the operator-splitting scheme in Section~\ref{sec:operator_splitting} applied to the heat equation and the convection-diffusion equation with the cell average limiter $\mathcal{P}_h$.
Our goal is to understand the error behavior of the optimization-based cell average limiter alone, without Zhang--Shu scaling postprocessing.
\par
The bilinear form $a_\mathcal{D}$ is coercive \cite{riviere2008discontinuous} for sufficiently large penalty parameter $\sigma$, namely there exists $\kappa > 0$, such that $\forall v_h \in V_h^0$,
\begin{align}\label{eq:coercivity}
\kappa\norm{v_h}{\mathrm{DG}}^2 \leq a_\mathcal{D}(v_h, v_h).
\end{align}
For SIPG, we can set $\kappa = 1/2$.
The DG form associated to the advection operator $\vec{b}\cdot\grad{u}$ satisfies the following positivity.
\begin{lemma}\label{lem:adv_pos}
If $\div\vec{b} = 0$ in $\Omega$ and $\vec{b}\cdot\normal = 0$ on $\partial\Omega$, then $a_{\mathrm{adv}}(v_h, v_h) \geq 0$ for all $v_h \in V_h$. We have
\begin{align*}
a_{\mathrm{adv}}(v_h, v_h) = \frac{1}{2}\sum_{e\in\Gammah}\int_e \abs{\vec{b}\cdot\normal_e} \jump{v_h}^2.
\end{align*}
\end{lemma}
\begin{proof}
On each element $E$, let $\normal_E$ denote the outer normal of $E$. Since $\vec{b}$ is divergence free, we have
\begin{align*}
-\int_E v_h\,\vec{b}\cdot\grad v_h = -\frac{1}{2}\int_{\partial E} v_h^2(\vec{b}\cdot\normal_E).
\end{align*}
Summing over all elements and collecting contributions face by face: on each interior face $e$, the element terms contribute $\frac{1}{2}[(v_h^+)^2-(v_h^-)^2](\vec{b}\cdot\normal_e)$.
If $\vec{b}\cdot\normal_e \geq 0$, $\upwind{v_h} = v_h^-$ and the face term contributes $v_h^-(v_h^- - v_h^+)(\vec{b}\cdot\normal_e)$; the sum equals $\frac{1}{2}(v_h^- - v_h^+)^2(\vec{b}\cdot\normal_e) = \frac{1}{2}\jump{v_h}^2\abs{\vec{b}\cdot\normal_e} \geq 0$.
If $\vec{b}\cdot\normal_e < 0$, $\upwind{v_h} = v_h^+$ and the face term contributes $v_h^+(v_h^- - v_h^+)(\vec{b}\cdot\normal_e)$; the sum equals $-\frac{1}{2}(v_h^+ - v_h^-)^2(\vec{b}\cdot\normal_e) = \frac{1}{2}\jump{v_h}^2\abs{\vec{b}\cdot\normal_e} \geq 0$.
Boundary face contributions vanish since $\vec{b}\cdot\normal = 0$ on $\partial\Omega$.
\end{proof}
The Poincar\'e inequality \cite{liu2023convergence} holds: there exists a constant $C_\mathrm{P} > 0$ independent of the mesh size $h$, such that $\forall v_h \in V_h$,
\begin{align*}
\norm{v_h}{L^2} \leq C_\mathrm{P}\Big(\norm{v_h}{\mathrm{DG}}^2 + \frac{1}{\abs{\Omega}}\Big|\int_\Omega v_h\Big|^2\Big)^{1/2}.
\end{align*}
Define the $L^2$ projection $\pi_h: L^2(\Omega) \to V_h$ as follows: for a given function $v \in L^2(\Omega)$ and any $E \in \setE_h$,
\begin{align*}
\int_E (\pi_h v - v) w_h = 0, \quad \forall w_h \in \IQ_k(E).
\end{align*}
Notice $1 \in \IQ_k(E)$, $\pi_h$ preserves cell averages. We have $\on{\overline{\pi_h v}}{E} = \on{\overline{v}}{E}$ for all $E\in\setE_h$.
In particular, if $u^0 \in [m,M]$, then the cell averages of $U_h^0 = \pi_h u^0$ lie in $[m,M]$ and the cell average limiter $\mathcal{P}_h$ acts as the identity on $U_h^0$. 
In the rest of this paper, we use $C$ to denote a generic constant that takes different values at different places and that is independent of mesh size $h$ and time step size $\tau$. 
The following approximation property holds: for $v \in H^{k+1}(\Omega)$,
\begin{align}\label{eq:L2_proj_approx}
\norm{\pi_h v - v}{L^2(E)} \leq Ch^{k+1}\abs{v}_{H^{k+1}(E)}
\quad\text{and}\quad
\abs{\pi_h v - v}_{H^s(E)} \leq Ch^{k+1-s}\abs{v}_{H^{k+1}(E)}.
\end{align}
Define the elliptic projection $\Pi_h: H^2(\setE_h) \to V_h$ as follows: for a given $v \in H^2(\setE_h)$ and any $v_h \in V_h^0$,
\begin{align*}
a_\mathcal{D}(v - \Pi_h v, v_h) = 0 \quad\text{subject to}\quad (\Pi_h v, 1) = (v, 1).
\end{align*}
The following approximation can be obtained from the error analysis for elliptic problems on convex domains \cite{riviere2008discontinuous}:
for $v \in H^{k+1}(\setE_h)$,
\begin{align}\label{eq:elliptic_proj_approx}
\norm{v - \Pi_h v}{L^2} + h\norm{v - \Pi_h v}{\mathrm{DG}} &\leq C h^{k+1} \norm{v}{H^{k+1}(\setE_h)}.
\end{align}
\par
For simplicity, we adopt the notation $u^n = u(t^n)$. Let $\theta_h^n = u_h^n - \Pi_h u^n$ denote the discrete error before postprocessing. We decompose the error as $u^n - U_h^n = \zeta^n - \Theta_h^n$, where the approximation error $\zeta^n = u^n - \Pi_h u^n$. The $\Theta_h^n = U_h^n - \Pi_h u^n$ denotes the discrete error after postprocessing. 
\begin{theorem}\label{thm:Theta_bound}
The approximation error and the discrete errors before and after postprocessing satisfy the following bound
\begin{align}\label{eq:Theta_bound}
\norm{\Theta_h^n}{L^2}^2 \leq \norm{\theta_h^n}{L^2}^2 + \norm{\overline{\zeta^n}}{L^2}^2 \leq \norm{\theta_h^n}{L^2}^2 + \norm{\zeta^n}{L^2}^2.
\end{align}
\end{theorem}
\begin{proof}
Define $\eta_h^n = (u_h^n - \overline{u_h^n}) - (\Pi_h u^n - \overline{\Pi_h u^n})$. Recall $U_h^n = \mathcal{P}_h u_h^n$, we have
$\Theta_h^n = \eta_h^n + (\mathcal{Q}\overline{u_h^n} - \overline{\Pi_h u^n})$ and
$\theta_h^n = \eta_h^n + (\overline{u_h^n} - \overline{\Pi_h u^n})$.
Notice that $\eta_h^n$ has vanishing cell average on every element, while the $\mathcal{Q}\overline{u_h^n} - \overline{\Pi_h u^n}$ and $\overline{u_h^n} - \overline{\Pi_h u^n}$ are piecewise constant. Thus, they are $L^2$ orthogonal in each identity. 
Define $x = \overline{u_h^n}$ and $y = \overline{\Pi_h u^n}$, the Pythagorean identity gives
\begin{subequations}\label{eq:Theta_pythagoras}
\begin{align}
\norm{\Theta_h^n}{L^2}^2 &= \norm{\eta_h^n}{L^2}^2 + \norm{\mathcal{Q}x - y}{L^2}^2,\\
\norm{\theta_h^n}{L^2}^2 &= \norm{\eta_h^n}{L^2}^2 + \norm{x - y}{L^2}^2.
\end{align}
\end{subequations}
Let $z = \overline{u^n}$. By Theorem~\ref{thm:mass_conservation}, $z \in \Lambda$ and $\overline{\zeta^n} = z - y$. Adding and subtracting $z$ in $\mathcal{Q}x - y$ and $x - y$, we have
\begin{align*}
\norm{\mathcal{Q}x - y}{L^2}^2 &= \norm{\mathcal{Q}x - z}{L^2}^2 + 2(\mathcal{Q}x - z,\, z - y) + \norm{z - y}{L^2}^2, \\
\norm{x - y}{L^2}^2 &= \norm{x - z}{L^2}^2 + 2(x - z,\, z - y) + \norm{z - y}{L^2}^2.
\end{align*}
Let $r = \mathcal{Q}x - x$ and $s = z - y$. Subtracting the second identity above from the first, we obtain
\begin{align*}
\norm{\mathcal{Q}x - y}{L^2}^2 - \norm{x - y}{L^2}^2
= \norm{\mathcal{Q}x - z}{L^2}^2 - \norm{x - z}{L^2}^2 + 2(r, s).
\end{align*}
Recall $z\in\Lambda$, applying Lemma~\ref{lem:proj_firm} yields $\norm{\mathcal{Q}x - z}{L^2}^2 - \norm{x - z}{L^2}^2 \leq -\norm{r}{L^2}^2$. By Cauchy--Schwarz's inequality, we have
\begin{align*}
\norm{\mathcal{Q}x - y}{L^2}^2 - \norm{x - y}{L^2}^2 \leq -\norm{r}{L^2}^2 + 2(r, s) \leq \norm{s}{L^2}^2 = \norm{\overline{\zeta^n}}{L^2}^2.
\end{align*}
Taking cell averages is the $L^2$ orthogonal projection onto $\IP_0(\setE_h)$, the decomposition \eqref{eq:cell_avg_decomposition} applied to $\zeta^n$ gives $\norm{\overline{\zeta^n}}{L^2} \leq \norm{\zeta^n}{L^2}$. Combining this with \eqref{eq:Theta_pythagoras}, we conclude the proof.
\end{proof}

\subsection{Example~1. Heat equation}
Let us derive an error estimate for the heat equation with implicit Euler time discretization and interior penalty DG space discretization.
Set $u_h^0 = U_h^0 = \pi_h u^0$, where $\pi_h$ denotes the $L^2$ projection onto $V_h$. Since $u^0$ lies in the admissible set $[m, M]$, so does the cell average of $U_h^0$, and no postprocessing is needed at the initial step. 
For $n = 0, 1, \cdots, \Nst - 1$, we compute $U_h^{n+1}$ from $U_h^n$ in two steps.
\begin{itemize}
\item[] Step 1 (Solve subproblem $\mathcal{H}$).~ Given $U_h^n\in V_h$, compute $u_h^{n+1} \in V_h$ such that
\begin{align}\label{eq:implicit_euler_heat}
(u_h^{n+1} - U_h^n, v_h) + \tau a_\mathcal{D}(u_h^{n+1}, v_h) = 0, && \forall v_h \in V_h.
\end{align}
\item[] Step 2 (Solve subproblem $\mathcal{L}$).~ Apply optimization-based cell average limiter $\mathcal{P}_h$
\begin{align*}
U_h^{n+1} = \mathcal{P}_h u_h^{n+1}.
\end{align*}
\end{itemize}

We have the following error bound.
\begin{theorem}\label{thm:error_bound_heat}
Assume $u \in L^\infty(0,T; H^{k+1})$, $\partial_t{u} \in L^2(0,T; H^{k+1})$, and $\partial_{tt} u \in L^2(0,T; L^2)$. For $k\geq 1$, prescribe the condition $\tau \geq ch^{2(k+1)/3}$. Then there exists $C > 0$ independent of $h$ and $\tau$, such that
\begin{align*}
\max_{0 \leq n \leq \Nst} \norm{U_h^n - u^n}{L^2} \leq C(\tau + h^{k+1}).
\end{align*}
\end{theorem}
\begin{proof}
The DG discretization is consistent, in the sense that the exact solution $u$ satisfies for $n\geq0$,
\begin{align}\label{eq:consistency_heat}
((\partial_t u)^{n+1}, v_h) + a_\mathcal{D}(u^{n+1}, v_h) = 0, && \forall v_h\in V_h.
\end{align}
Subtracting \eqref{eq:implicit_euler_heat} from \eqref{eq:consistency_heat}, using the Taylor expansion of $(\partial_t u)^{n+1}$ with integral remainder, we get
\begin{align*}
(u^{n+1} - u_h^{n+1} - (u^n - U_h^n), v_h) + \tau a_\mathcal{D}(u^{n+1} - u_h^{n+1}, v_h) = \int_{t^n}^{t^{n+1}} (t^n - t)(\partial_{tt} u, v_h), && \forall v_h\in V_h.
\end{align*}
By definition, we have $u^{n+1} - u_h^{n+1} = \zeta^{n+1} - \theta_h^{n+1}$ and $u^n - U_h^n = \zeta^n - \Theta_h^n$. Using $a_\mathcal{D}(\zeta^{n+1}, v_h) = 0$, the above reduces to
\begin{align}\label{eq:theta_eq_heat}
(\theta_h^{n+1} - \Theta_h^n, v_h) + \tau a_\mathcal{D}(\theta_h^{n+1}, v_h)
= (\zeta^{n+1} - \zeta^n, v_h) + \int_{t^n}^{t^{n+1}} (t-t^n)(\partial_{tt} u, v_h), && \forall v_h \in V_h^0.
\end{align}
Since $u_h^{n+1}$ and $\Pi_h u^{n+1}$ have the same total mass, by Theorem~\ref{thm:mass_conservation} and the definition of $\Pi_h$, we take $v_h = \theta_h^{n+1}\in V_h^0$ in \eqref{eq:theta_eq_heat}. Then, the inequality $(a-b)a \geq \frac{1}{2}(a^2-b^2)$ and coercivity \eqref{eq:coercivity} implies
\begin{align}\label{eq:energy_ineq_heat}
\frac{1}{2}(\norm{\theta_h^{n+1}}{L^2}^2 - \norm{\Theta_h^n}{L^2}^2) + \tau\kappa\norm{\theta_h^{n+1}}{\mathrm{DG}}^2
\leq (\zeta^{n+1} - \zeta^n, \theta_h^{n+1}) + \int_{t^n}^{t^{n+1}} (t-t^n)(\partial_{tt} u, \theta_h^{n+1}).
\end{align}
Since $u^n \in [m,M]$, together with Theorem~\ref{thm:mass_conservation}, we have $\overline{u^n}\in\Lambda$. Thus, Theorem~\ref{thm:Theta_bound} applies and yields \eqref{eq:Theta_bound}. Substituting \eqref{eq:Theta_bound} into \eqref{eq:energy_ineq_heat}, we get
\begin{align}\label{eq:after_sub_heat}
\frac{1}{2}\norm{\theta_h^{n+1}}{L^2}^2 + \tau\kappa\norm{\theta_h^{n+1}}{\mathrm{DG}}^2
&\leq \frac{1}{2}\norm{\theta_h^n}{L^2}^2 + \frac{1}{2}\norm{\zeta^n}{L^2}^2 \nonumber\\
&+ (\zeta^{n+1}-\zeta^n, \theta_h^{n+1})
+ \int_{t^n}^{t^{n+1}} (t-t^n)(\partial_{tt}u, \theta_h^{n+1}).
\end{align}
Applying Cauchy--Schwarz's, Poincar\'e, and Young's inequalities to the third term on the right-hand side, followed by using \eqref{eq:elliptic_proj_approx} and $\norm{u^{n+1}-u^n}{H^{k+1}} \leq \tau^{1/2}\norm{\partial_t u}{L^2(t^n,t^{n+1};H^{k+1})}$, we get
\begin{align}\label{eq:heat_bound_rhs1}
\abs{(\zeta^{n+1}-\zeta^n,\, \theta_h^{n+1})} 
&\leq \frac{C_\mathrm{P}^2}{\tau\kappa}\norm{\zeta^{n+1}-\zeta^n}{L^2}^2 + \frac{\kappa\tau}{4}\norm{\theta_h^{n+1}}{\mathrm{DG}}^2 \nonumber\\
&\leq \frac{\kappa\tau}{4}\norm{\theta_h^{n+1}}{\mathrm{DG}}^2 + \frac{Ch^{2k+2}}{\kappa}\norm{\partial_t u}{L^2(t^n,t^{n+1};H^{k+1})}^2.
\end{align}
For the last term in \eqref{eq:after_sub_heat}, applying Cauchy--Schwarz's, Poincar\'e, and Young's inequalities again, we have
\begin{align}\label{eq:heat_bound_rhs2}
\Big|\int_{t^n}^{t^{n+1}}(t-t^n)(\partial_{tt}u,\,\theta_h^{n+1})\Big|
&\leq \frac{\tau^{3/2}}{\sqrt{3}} \norm{\partial_{tt}u}{L^2(t^n,t^{n+1};L^2)}\norm{\theta_h^{n+1}}{L^2} \nonumber\\
&\leq \frac{C_\mathrm{P}^2\tau^2}{3\kappa}
\norm{\partial_{tt}u}{L^2(t^n,t^{n+1}; L^2)}^2
+ \frac{\kappa\tau}{4}\norm{\theta_h^{n+1}}{\mathrm{DG}}^2.
\end{align}
Substituting \eqref{eq:heat_bound_rhs1} and \eqref{eq:heat_bound_rhs2} into \eqref{eq:after_sub_heat}, we have
\begin{multline}\label{eq:recursion_heat}
\frac{1}{2}\norm{\theta_h^{n+1}}{L^2}^2 - \frac{1}{2}\norm{\theta_h^n}{L^2}^2 + \frac{\tau\kappa}{2}\norm{\theta_h^{n+1}}{\mathrm{DG}}^2 \leq \frac{1}{2}\norm{\zeta^n}{L^2}^2 \\
+ \frac{Ch^{2k+2}}{\kappa}\norm{\partial_t u}{L^2(t^n,t^{n+1};H^{k+1})}^2 + \frac{C\tau^2}{\kappa}\norm{\partial_{tt}u}{L^2(t^n,t^{n+1}; L^2)}^2.
\end{multline}
For any integer $1\leq N\leq \Nst$, summing \eqref{eq:recursion_heat} over $n = 0, \cdots, N-1$, we obtain
\begin{multline}\label{eq:1_heat}
\frac{1}{2}\norm{\theta_h^N}{L^2}^2 + \frac{\tau\kappa}{2}\sum_{n=0}^{N-1}\norm{\theta_h^{n+1}}{\mathrm{DG}}^2
\leq \frac{1}{2}\norm{\theta_h^0}{L^2}^2 + \frac{1}{2}\sum_{n=0}^{N-1}\norm{\zeta^n}{L^2}^2 \\
+ \frac{Ch^{2k+2}}{\kappa}\norm{\partial_t u}{L^2(0,T;H^{k+1})}^2
+ \frac{C\tau^2}{\kappa}\norm{\partial_{tt}u}{L^2(0,T;L^2)}^2.
\end{multline}
It remains to bound the accumulated approximation error. Using the property of elliptic projection \eqref{eq:elliptic_proj_approx} and $N \leq T/\tau$, we have
\begin{align}\label{eq:heat_accumulated_approx_err}
\sum_{n=0}^{N-1}\norm{\zeta^n}{L^2}^2 \leq \frac{CT}{\tau}h^{2k+2}\norm{u}{L^\infty(0,T;H^{k+1})}^2.
\end{align}
The condition $\tau \geq ch^{2(k+1)/3}$ implies $h^{2k+2}\leq C\tau^3$. Thus, substituting \eqref{eq:heat_accumulated_approx_err} into \eqref{eq:1_heat}, using $\norm{\theta_h^0}{L^2}^2\leq Ch^{2k+2}$, for any integer $1\leq N\leq \Nst$, we obtain
\begin{align}\label{eq:L2_bound_theta_h_heat}
\norm{\theta_h^N}{L^2}^2 + \tau\kappa\sum_{n=0}^{N-1}\norm{\theta_h^{n+1}}{\mathrm{DG}}^2 \leq C(\tau^2 + h^{2k+2}).
\end{align}
By error decomposition $u^n - U_h^n = \zeta^n - \Theta_h^n$, Theorem~\ref{thm:limiter_accuracy}, and the triangle inequality, we have
\begin{align*}
\norm{u^n - U_h^n}{L^2} = \norm{\zeta^n - \Theta_h^n}{L^2} \leq \norm{\zeta^n - \theta_h^n}{L^2} \leq \norm{\theta_h^n}{L^2} + \norm{\zeta^n}{L^2}.
\end{align*}
From \eqref{eq:L2_bound_theta_h_heat}, we know $\max_n\norm{\theta_h^n}{L^2} \leq C(\tau + h^{k+1})$. In addition, \eqref{eq:elliptic_proj_approx} gives
$\norm{\zeta^n}{L^2} \leq Ch^{k+1}\norm{u(t^n)}{H^{k+1}}$.
Combining these together, we conclude the proof.
\end{proof}

\subsection{Example~2. Convection-diffusion equation}
Let us derive an error estimate for the convection-diffusion equation with implicit Euler time discretization and interior penalty DG space discretization. 
Set $u_h^0 = U_h^0 = \pi_h u^0$, where $\pi_h$ denotes the $L^2$ projection onto $V_h$. Since $u^0$ lies in the admissible set $[m, M]$, so does the cell average of $U_h^0$, and no postprocessing is needed at the initial step. 
For $n = 0, 1, \cdots, \Nst - 1$, we compute $U_h^{n+1}$ from $U_h^n$ in two steps.
\begin{itemize}
\item[] Step 1 (Solve subproblem $\mathcal{H}$).~ Given $U_h^n\in V_h$, compute $u_h^{n+1} \in V_h$ such that
\begin{align}\label{eq:implicit_euler_conv_diff}
(u_h^{n+1} - U_h^n, v_h) + \tau D a_\mathcal{D}(u_h^{n+1}, v_h) + \tau a_{\mathrm{adv}}(u_h^{n+1}, v_h) = 0, && \forall v_h \in V_h.
\end{align}
\item[] Step 2 (Solve subproblem $\mathcal{L}$).~ Apply optimization-based cell average limiter $\mathcal{P}_h$
\begin{align*}
U_h^{n+1} = \mathcal{P}_h u_h^{n+1}.
\end{align*}
\end{itemize}

We have the following error bound.
\begin{theorem}\label{thm:error_bound_conv_diff}
Given a divergence-free field $\vec{b}$ with vanishing normal component on the boundary, i.e., $\vec{b}\cdot\normal = 0$.
Assume $u \in L^\infty(0,T; H^{k+1})$, $\partial_t{u} \in L^2(0,T; H^{k+1})$, and $\partial_{tt} u \in L^2(0,T; L^2)$. For $k\geq 1$, prescribe the condition $\tau \geq ch^{2(k+1)/3}$. Then there exists $C > 0$ independent of $h$ and $\tau$, such that
\begin{align*}
\max_{0 \leq n \leq \Nst} \norm{U_h^n - u^n}{L^2} \leq C(\tau + h^{k+1}).
\end{align*}
\end{theorem}
\begin{proof}
The DG discretization is consistent, in the sense that the exact solution $u$ satisfies for $n\geq0$,
\begin{align}\label{eq:conv_diff_consistent}
((\partial_t u)^{n+1}, v_h) + D a_\mathcal{D}(u^{n+1}, v_h) + a_{\mathrm{adv}}(u^{n+1}, v_h) = 0, && \forall v_h\in V_h.
\end{align}
Proceeding analogously to the proof of Theorem~\ref{thm:error_bound_heat}, we subtract \eqref{eq:implicit_euler_conv_diff} from \eqref{eq:conv_diff_consistent} and use the Taylor expansion of $(\partial_t u)^{n+1}$ with integral remainder. 
Recalling the definitions of the approximation error $\zeta^n$ and the discrete errors $\theta_h^n$ and $\Theta_h^n$, and using $a_\mathcal{D}(\zeta^{n+1},v_h)=0$, we have
\begin{multline}\label{eq:theta_eq_conv_diff}
(\theta_h^{n+1} - \Theta_h^n, v_h) + \tau D a_\mathcal{D}(\theta_h^{n+1}, v_h) + \tau a_{\mathrm{adv}}(\theta_h^{n+1}, v_h) \\
= (\zeta^{n+1} - \zeta^n, v_h) + \tau a_{\mathrm{adv}}(\zeta^{n+1}, v_h) + \int_{t^n}^{t^{n+1}} (t-t^n)(\partial_{tt} u, v_h),\quad \forall v_h\in V_h^0.
\end{multline}
Choosing $v_h = \theta_h^{n+1}\in V_h^0$ in \eqref{eq:theta_eq_conv_diff}, from the coercivity of $a_\mathcal{D}$, the positivity of $a_\mathrm{adv}$ in Lemma~\ref{lem:adv_pos}, and Theorem~\ref{thm:Theta_bound}, we have
\begin{multline}\label{eq:after_sub_conv_diff}
\frac{1}{2}\norm{\theta_h^{n+1}}{L^2}^2 + \tau D\kappa\norm{\theta_h^{n+1}}{\mathrm{DG}}^2
\leq \frac{1}{2}\norm{\theta_h^n}{L^2}^2 + \frac{1}{2}\norm{\zeta^n}{L^2}^2 \\
+ (\zeta^{n+1}-\zeta^n, \theta_h^{n+1}) + \tau a_{\mathrm{adv}}(\zeta^{n+1}, \theta_h^{n+1})
+ \int_{t^n}^{t^{n+1}} (t-t^n)(\partial_{tt}u, \theta_h^{n+1}).
\end{multline}
Applying Cauchy--Schwarz's, Poincar\'e, and Young's inequalities to the third term on the right-hand side, followed by using \eqref{eq:elliptic_proj_approx} and $\norm{u^{n+1}-u^n}{H^{k+1}} \leq \tau^{1/2}\norm{\partial_t u}{L^2(t^n,t^{n+1};H^{k+1})}$, we get
\begin{align}\label{eq:conv_diff_bound_rhs1}
\abs{(\zeta^{n+1}-\zeta^n, \theta_h^{n+1})} 
\leq \frac{\tau D\kappa}{4}\norm{\theta_h^{n+1}}{\mathrm{DG}}^2 + \frac{Ch^{2k+2}}{D\kappa} \norm{\partial_t u}{L^2(t^n,t^{n+1}; H^{k+1})}^2.
\end{align}
For the fourth term on the right-hand side of \eqref{eq:after_sub_conv_diff}, we bound the volume and face parts separately. The volume part gives
\begin{align*}
\Big|-\sum_{E\in\setE_h}\int_E \zeta^{n+1} \vec{b}\cdot\grad\theta_h^{n+1}\Big| 
\leq \norm{\vec{b}}{L^\infty}\norm{\zeta^{n+1}}{L^2}\norm{\theta_h^{n+1}}{\mathrm{DG}}
\leq Ch^{k+1}\norm{\theta_h^{n+1}}{\mathrm{DG}}.
\end{align*}
For the face part, extracting $\abs{\vec{b}\cdot\normal_e}\leq\norm{\vec{b}}{L^\infty}$, applying Cauchy--Schwarz's and trace inequalities, and recalling the definition of DG norm, we have
\begin{align*}
\Big|\sum_{e\in\Gammah}\int_e\upwind{(\zeta^{n+1})}(\vec{b}\cdot\normal_e)\jump{\theta_h^{n+1}}\Big|
&\leq \norm{\vec{b}}{L^\infty}\sum_{e\in\Gammah}\norm{\upwind{(\zeta^{n+1})}}{L^2(e)}\norm{\jump{\theta_h^{n+1}}}{L^2(e)} \\
&\leq \norm{\vec{b}}{L^\infty}\Big(\sum_e h\norm{\upwind{(\zeta^{n+1})}}{L^2(e)}^2\Big)^{1/2} \Big(\sum_e h^{-1}\norm{\jump{\theta_h^{n+1}}}{L^2(e)}^2\Big)^{1/2}\\
&\leq C h^{k+1} \norm{\theta_h^{n+1}}{\mathrm{DG}}.
\end{align*}
Therefore, combining the volume and face bounds above and applying Young's inequality, we obtain
\begin{align}\label{eq:conv_diff_bound_rhs2}
\tau\abs{a_{\mathrm{adv}}(\zeta^{n+1}, \theta_h^{n+1})}
\leq C\tau h^{k+1}\norm{\theta_h^{n+1}}{\mathrm{DG}}
\leq \frac{\tau D\kappa}{4}\norm{\theta_h^{n+1}}{\mathrm{DG}}^2 + \frac{C\tau h^{2k+2}}{D\kappa}.
\end{align}
For the last term in \eqref{eq:after_sub_conv_diff}, applying Cauchy--Schwarz's, Poincar\'e, and Young's inequalities again, we get
\begin{align}\label{eq:conv_diff_bound_rhs3}
\Big|\int_{t^n}^{t^{n+1}}(t-t^n)(\partial_{tt}u, \theta_h^{n+1})\Big| \leq \frac{\tau D\kappa}{4}\norm{\theta_h^{n+1}}{\mathrm{DG}}^2 + \frac{C \tau^2}{D\kappa}\norm{\partial_{tt}u}{L^2(t^n,t^{n+1}; L^2)}^2.
\end{align}
Substituting \eqref{eq:conv_diff_bound_rhs1}, \eqref{eq:conv_diff_bound_rhs2}, and \eqref{eq:conv_diff_bound_rhs3} into \eqref{eq:after_sub_conv_diff}, we have
\begin{multline}\label{eq:recursion}
\frac{1}{2}\norm{\theta_h^{n+1}}{L^2}^2 - \frac{1}{2}\norm{\theta_h^n}{L^2}^2 + \frac{\tau D\kappa}{4}\norm{\theta_h^{n+1}}{\mathrm{DG}}^2 \leq \frac{1}{2}\norm{\zeta^n}{L^2}^2 \\
+ \frac{C\tau h^{2k+2}}{D\kappa} + \frac{Ch^{2k+2}}{D\kappa}\norm{\partial_t u}{L^2(t^n,t^{n+1};H^{k+1})}^2 + \frac{C\tau^2}{D\kappa} \norm{\partial_{tt}u}{L^2(t^n,t^{n+1}; L^2)}^2.
\end{multline}
For any integer $1\leq N\leq \Nst$, summing \eqref{eq:recursion} over $n = 0, \cdots, N-1$, we obtain
\begin{multline*}
\frac{1}{2}\norm{\theta_h^N}{L^2}^2 + \frac{\tau D\kappa}{4}\sum_{n=0}^{N-1}\norm{\theta_h^{n+1}}{\mathrm{DG}}^2
\leq \frac{1}{2}\norm{\theta_h^0}{L^2}^2 + \frac{1}{2}\sum_{n=0}^{N-1}\norm{\zeta^n}{L^2}^2
+ \frac{CTh^{2k+2}}{D\kappa} \\
+ \frac{Ch^{2k+2}}{D\kappa}\norm{\partial_t u}{L^2(0,T;H^{k+1})}^2
+ \frac{C\tau^2}{D\kappa}\norm{\partial_{tt}u}{L^2(0,T;L^2)}^2.
\end{multline*}
As in the proof of Theorem~\ref{thm:error_bound_heat}, the second term on the right-hand side above satisfies \eqref{eq:heat_accumulated_approx_err}.
Since the condition $\tau \geq ch^{2(k+1)/3}$ implies $h^{2k+2}\leq C\tau^3$, using $\norm{\theta_h^0}{L^2}^2\leq Ch^{2k+2}$, for any integer $1\leq N\leq \Nst$, we obtain
\begin{align}\label{eq:L2_bound_theta_h}
\norm{\theta_h^N}{L^2}^2 + \frac{\tau D\kappa}{2}\sum_{n=0}^{N-1}\norm{\theta_h^{n+1}}{\mathrm{DG}}^2 \leq C(\tau^2 + h^{2k+2}).
\end{align}
By the error decomposition $u^n - U_h^n = \zeta^n - \Theta_h^n$, Theorem~\ref{thm:limiter_accuracy}, and the triangle inequality, we get $\norm{u^n - U_h^n}{L^2} \leq \norm{\theta_h^n}{L^2} + \norm{\zeta^n}{L^2}$.
From \eqref{eq:L2_bound_theta_h}, we know $\max_n\norm{\theta_h^n}{L^2} \leq C(\tau + h^{k+1})$. In addition, \eqref{eq:elliptic_proj_approx} gives
$\norm{\zeta^n}{L^2} \leq Ch^{k+1}\norm{u(t^n)}{H^{k+1}}$.
Combining these together, we conclude the proof.
\end{proof}
In addition to the conditional optimal convergence under a reverse CFL condition $\tau \geq ch^{2(k+1)/3}$, the scheme enjoys unconditional suboptimal convergence as follows.
\begin{theorem}\label{thm:error_bound_unconditional}
Given a divergence-free field $\vec{b}$ with vanishing normal component on the boundary. Assume $u \in L^\infty(0,T; H^{k+1})$, $\partial_t{u} \in L^2(0,T; L^2)$, and $\partial_{tt}{u} \in L^2(0,T; L^2)$. For $k\geq 1$, there exists $C > 0$ independent of $h$ and $\tau$, such that
\begin{align*}
\max_{0 \leq n \leq \Nst} \norm{U_h^n - u^n}{L^2} \leq C(\tau + h^k).
\end{align*}
\end{theorem}
\begin{proof}
The proof parallels that of Theorem~\ref{thm:error_bound_conv_diff}, with elliptic projection replaced by $L^2$ projection, see Appendix~\ref{app:unconditional_subopt}.
\end{proof}

\section{Numerical experiments}\label{sec:numerical_experiments}
In this section, we verify that our scheme preserves both conservation and bounds with high-order accuracy in space. We briefly review existing splitting methods and present an efficient algorithm for the cell average limiter. Convergence rate studies for the heat equation and the convection-diffusion equation are presented using manufactured solutions that trigger the limiter. A solid body rotation test is then examined to validate the conservative bound-preserving property. 
\par 
The penalty parameters for all tests are as follows: we use SIPG with penalty parameters $\sigma = 4$ in $\IQ_2$ space and $\sigma = 8$ in $\IQ_3$ space. 
We set the convergence tolerance for the iterative optimization solver to $\epsilon = 10^{-10}$. 
In all tests, the Zhang--Shu scaling limiter is applied after the optimization-based cell average limiter to enforce the bounds at the quadrature points. It modifies only the high-order moments and thus preserves the cell averages and global conservation.
\begin{remark}
The accuracy of the Zhang--Shu scaling limiter, provided the cell averages are admissible, follows from \cite{zhang2017positivity,xu2017bound}. The present analysis in Section~\ref{sec:error_analysis} concerns repeated application of the conservative cell average correction. 
It does not establish an error estimate for the additional Zhang--Shu scaling applied at every time step, which is also beyond the scope of this paper.
\end{remark}

\subsection{Efficient splitting solver}
An efficient optimization solver is essential for the cell average limiter to be practical, because the problem \eqref{eq:opt_model1} may need to be solved at every time step. Beyond efficiency, the solver should also be easy to implement and highly amenable to parallelization. In this subsection, we briefly review existing iterative solvers, then present an efficient splitting algorithm.
\par
Splitting methods trace back to the 1950s, when Peaceman and Rachford~\cite{peaceman1955numerical} as well as Douglas and Rachford~\cite{douglas1956numerical} introduced schemes for the heat equation.
Lions and Mercier~\cite{lions1979splitting} extended these schemes to the sum of two maximal monotone operators, giving the generalized Douglas--Rachford (DR) splitting method. 
The DR method is closely related to several other popular first-order methods, including the alternating direction method of multipliers~\cite{fortin2000augmented}, the primal-dual hybrid gradient method~\cite{chambolle2016introduction}, and the split Bregman method~\cite{goldstein2009split}.  We refer to~\cite{demanet2016eventual,lai2026ubiquitous} and the references therein for the equivalences among these schemes.
\par
The DR method can be integrated with forward-backward splitting~\cite{raguet2019note,raguet2013generalized}, an approach that was subsequently extended in~\cite{briceno2015forward}. The convergence of this method was established by Davis and Yin~\cite{davis2017three}, and it is commonly referred to as the Davis--Yin (DY) splitting.
For $\min_x d_1(x) + d_2(x) + d_3(x)$ with $\grad{d_2}$ Lipschitz continuous, the DY iteration reads:
\begin{align*}
\text{(DY)}~
\begin{cases}
x^{k+1/2} &\hspace{-0.75em}= \prox_{d_3}^\gamma(z^k),\\
x^{k+1} &\hspace{-0.75em}= \prox_{d_1}^\gamma(2x^{k+1/2} - z^{k} - \gamma\grad{d_2}(x^{k+1/2})),\\
z^{k+1} &\hspace{-0.75em}= z^{k} + x^{k+1} - x^{k+1/2}.
\end{cases}
\end{align*}
Let $\vec{x}$ denote a vector that stores the cell averages of the piecewise constant function $x_h$ on mesh $\setE_h$. 
We apply the change of variables $[\vec{X}]_i = X_i = \sqrt{\abs{E_i}}x_i$ and $[\vec{W}]_i = W_i = \sqrt{\abs{E_i}}\on{\overline{u_h}}{E_i}$. Define the matrix $\vecc{A} = [\sqrt{\abs{E_1}}, \cdots, \sqrt{\abs{E_\Nel}}]\in\IR^{1\times\Nel}$ and the scalar $b = \vecc{A}\vec{W}$. Then, the problem \eqref{eq:opt_model1} is equivalent to
\begin{align}\label{eq:optimization_model_vec_form}
\min_{\vec{X}\in\IR^{\Nel}} \norm{\vec{X}-\vec{W}}{2}^2 + \iota_{\Lambda_1}(\vec{X}) + \iota_{\Lambda_2}(\vec{X}),
\end{align}
where the sets $\Lambda_1$ and $\Lambda_2$ are reused to denote their vector-space counterparts, as follows:
\begin{align*}
\Lambda_1 = \{\vec{X}\in\IR^\Nel: \vecc{A}\vec{X} = b\}
\quad\text{and}\quad
\Lambda_2 = \{\vec{X}\in\IR^\Nel: \sqrt{\abs{E_i}}\,m \leq X_i \leq \sqrt{\abs{E_i}}\,M,~\forall i\}.
\end{align*}
Once the minimizer $\vec{X}^\ast$ of \eqref{eq:optimization_model_vec_form} is obtained, the limited cell averages are recovered by $x_i^\ast = X_i^\ast/\sqrt{\abs{E_i}}$.
Choosing $d_1(\vec{X}) = \iota_{\Lambda_1}(\vec{X})$, $d_2(\vec{X}) = \norm{\vec{X}-\vec{W}}{2}^2$, and $d_3(\vec{X}) = \iota_{\Lambda_2}(\vec{X})$; and taking $\gamma = 1/2$ (since $\grad{d_2}$ is 2-Lipschitz), the DY iteration specializes to
\begin{subequations}\label{eq:DY_nonuniform}
\begin{align}
[\vec{X}^{k+1/2}]_i &= \min(\max([\vec{Z}^k]_i, \sqrt{\abs{E_i}}m), \sqrt{\abs{E_i}}M),\quad\text{for}~i = 1,\cdots,\Nel, \\
\vec{Y}^{k+1} &= \vec{X}^{k+1/2} - \vec{Z}^{k} + \vec{W}, \\
\vec{X}^{k+1} &= \vecc{A}^{+}(b - \vecc{A}\vec{Y}^{k+1}) + \vec{Y}^{k+1}, \\
\vec{Z}^{k+1} &= \vec{Z}^{k} + \vec{X}^{k+1} - \vec{X}^{k+1/2}.
\end{align}
\end{subequations}
Here, the pseudoinverse $\vecc{A}^+ = \transpose{\vecc{A}}(\vecc{A}\transpose{\vecc{A}})^{-1} = \frac{1}{\abs{\Omega}}\transpose{\vecc{A}}$.
\par
The algorithm \eqref{eq:DY_nonuniform} is straightforward to implement and parallel-friendly, requiring only vectorizable operations (cut-off, scaling, and addition) and a single vector inner product. Through its KKT conditions, \eqref{eq:opt_model1} also admits an exact finite solution by a one-dimensional breakpoint search, at the cost of a sort or selection over elements, which is harder to implement on distributed-memory systems. We therefore favor \eqref{eq:DY_nonuniform}. On a uniform mesh, this iteration recovers the formula of our prior work \cite{liu2025efficient}.

\subsection{Convergence study for heat equation}
We solve the heat equation $u_t - \laplace{u} = f$ on domain $\Omega=(0,1)^2$ to final time $T = 1$ with homogeneous Neumann boundary conditions. The source term is computed from the following manufactured solution:
\begin{align}\label{eq:heat_prescribe_sol}
u(x,y,t) = \frac{1}{2}\Big(1 + \cos(3\pi x) + \nu (2t-1)\sin^2(3\pi x)\cos(4\pi y)\Big), \quad \text{where}~\nu = 0.45.
\end{align}
The above prescribed solution satisfies $u\in [0, 1]$, attaining $u = 0$ at $x = \frac{1}{3}$ and $u = 1$ at $x = 0$ for every $t\in [0,1]$. It has a constant spatial mean $\frac{1}{2}$. 
Hence, the admissible set is $[m,M] = [0,1]$ and the total mass is conserved, i.e., both the conservation constraint $\Lambda_1$ and the bound constraint $\Lambda_2$ are exercised. 
The computation is performed on a sequence of uniformly refined meshes, halving $h$ and $\tau$ together. Note that since the coefficient $2t-1$ in \eqref{eq:heat_prescribe_sol} is linear in time, the implicit Euler scheme commits no temporal error here. 
\par
Table~\ref{tab:heat_Q2} and Table~\ref{tab:heat_Q3} compare results with and without applying the optimization-based cell average limiter. 
For simulations in $\IQ_2$ and $\IQ_3$ spaces, we report the range of the discrete solution over all time steps, together with the largest excursion $\delta = \max\{\max u_h-M, m-\min u_h, 0\}$ outside the admissible set. 
Without the limiter, the discrete solution overshoots $M = 1$ and undershoots $m = 0$ by up to $2.678\times 10^{-3}$ on the coarsest mesh, and $\delta > 0$ at every resolution. 
With the limiter, $U_h\in [0,1]$ in every test. The optimal convergence rates are observed, which confirms that the limiter preserves accuracy.
\begin{table}[ht!]
\centering
\begin{tabular}{cccc ccc cc}
\toprule
 & & & &\multicolumn{3}{c}{without limiter} & \multicolumn{2}{c}{with limiter}\\
\cmidrule(lr){5-7}\cmidrule(lr){8-9}
$h$ & $\tau$ & {\scriptsize$\norm{U_h^{\Nst} \!\!-\! u(T)}{L_h^2}$} & rate
 & $\min u_h$ & $\max u_h$ & $\delta$
 & $\min U_h$ & $\max U_h$\\
\midrule
$2^{-3}$ & $2^0\cdot10^{-2}$ & $5.134\cdot10^{-3}$ & ---     & $-0.002678$ & $1.002678$ & $2.678\cdot10^{-3}$ & $0.000000$ & $1.000000$\\
$2^{-4}$ & $2^{-1}\cdot10^{-2}$ & $6.285\cdot10^{-4}$ & $3.030$ & $-0.000123$ & $1.000123$ & $1.229\cdot10^{-4}$ & $0.000000$ & $1.000000$\\
$2^{-5}$ & $2^{-2}\cdot10^{-2}$ & $7.824\cdot10^{-5}$ & $3.006$ & $-0.000007$ & $1.000007$ & $6.916\cdot10^{-6}$ & $0.000000$ & $1.000000$\\
$2^{-6}$ & $2^{-3}\cdot10^{-2}$ & $9.771\cdot10^{-6}$ & $3.001$ & $-0.000000$ & $1.000000$ & $4.202\cdot10^{-7}$ & $0.000000$ & $1.000000$\\
\bottomrule
\end{tabular}
\caption{Heat equation. Convergence and bound preservation without and with the cell average limiter in $\IQ_2$ space.}
\label{tab:heat_Q2}
\end{table}

\begin{table}[ht!]
\centering
\begin{tabular}{cccc ccc cc}
\toprule
 & & & &\multicolumn{3}{c}{without limiter} & \multicolumn{2}{c}{with limiter}\\
\cmidrule(lr){5-7}\cmidrule(lr){8-9}
$h$ & $\tau$ & {\scriptsize$\norm{U_h^{\Nst} \!\!-\! u(T)}{L_h^2}$} & rate
 & $\min u_h$ & $\max u_h$ & $\delta$
 & $\min U_h$ & $\max U_h$\\
\midrule
$2^{-3}$ & $2^0\cdot10^{-2}$ & $5.650\cdot10^{-4}$ & ---     & $-0.000031$ & $1.000031$ & $3.133\cdot10^{-5}$ & $0.000000$ & $1.000000$\\
$2^{-4}$ & $2^{-1}\cdot10^{-2}$ & $3.889\cdot10^{-5}$ & $3.861$ & $-0.000001$ & $1.000001$ & $7.689\cdot10^{-7}$ & $0.000000$ & $1.000000$\\
$2^{-5}$ & $2^{-2}\cdot10^{-2}$ & $2.507\cdot10^{-6}$ & $3.955$ & $-0.000000$ & $1.000000$ & $1.363\cdot10^{-8}$ & $0.000000$ & $1.000000$\\
$2^{-6}$ & $2^{-3}\cdot10^{-2}$ & $1.580\cdot10^{-7}$ & $3.988$ & $-0.000000$ & $1.000000$ & $2.244\!\cdot\!10^{-10}$ & $0.000000$ & $1.000000$\\
\bottomrule
\end{tabular}
\caption{Heat equation. Convergence and bound preservation without and with the cell average limiter in $\IQ_3$ space.}
\label{tab:heat_Q3}
\end{table}

\subsection{Rotating Gaussian hill}
A Gaussian hill rotating in a solenoidal velocity field serves as a widely used manufactured solution for convection-diffusion solvers. We utilize this classical benchmark for a convergence rate study.  
\par
The computational domain is $\Omega = [-2,2]^2$ with start time $t^0 = \pi/2$ and end time $T = 3\pi/2$. The divergence-free velocity field is $\vec{b} = \transpose{[-y,x]}$, and the diffusion coefficient is $D = 10^{-2}$. The initial and boundary conditions are prescribed by the following manufactured solution
\begin{align*}
u(x,y,t) = \frac{1}{4\pi D t} \exp\left(-\frac{(x+\frac{1}{2}\sin t)^2 + (y-\frac{1}{2}\cos t)^2}{4D t}\right).
\end{align*}
This solution describes a Gaussian hill that rotates counterclockwise while being gradually smeared by diffusion. The Gaussian decays exponentially away from its center, so on the finite domain $\Omega$ the no-flux condition \eqref{eq:conv_diff_BC} is satisfied up to an exponentially small residual. Figure~\ref{fig:Gaussian_hill} shows the initial and final time solutions on a $64\times 64$ mesh with $\IQ_3$ space discretization.
\par
We use uniform square meshes for both $\IQ_2$ and $\IQ_3$ spaces. For the $\IQ_2$ scheme, we use the tensor product of three-point Gauss--Legendre quadrature; for the $\IQ_3$ scheme, we use the tensor product of four-point Gauss--Legendre quadrature. 
The cell average limiter with $[m,M] = [0,+\infty)$ is triggered during the simulation. Table~\ref{tab:Gaussian_hill} shows the $L^2_h$ errors and convergence rates on a sequence of refined meshes. Optimal convergence rates, third order for the $\IQ_2$ scheme and fourth order for the $\IQ_3$ scheme, are obtained.
\begin{figure}[ht!]
\centering
\begin{tabularx}{0.95\linewidth}{@{}c@{~}c@{}}
\includegraphics[width=0.475\textwidth]{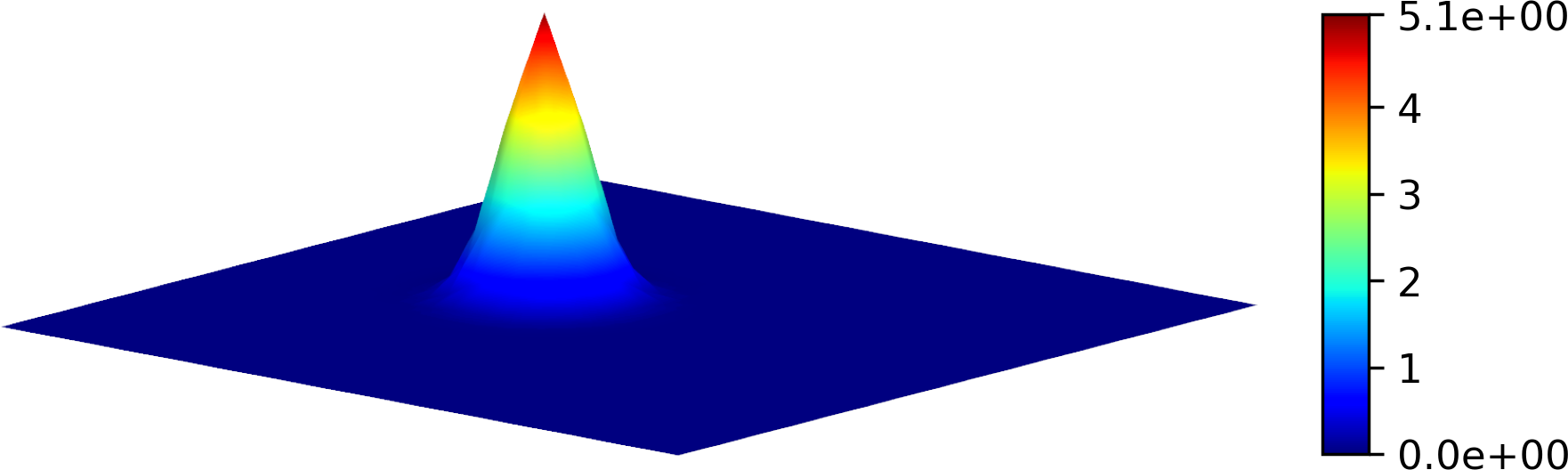} &
\includegraphics[width=0.475\textwidth]{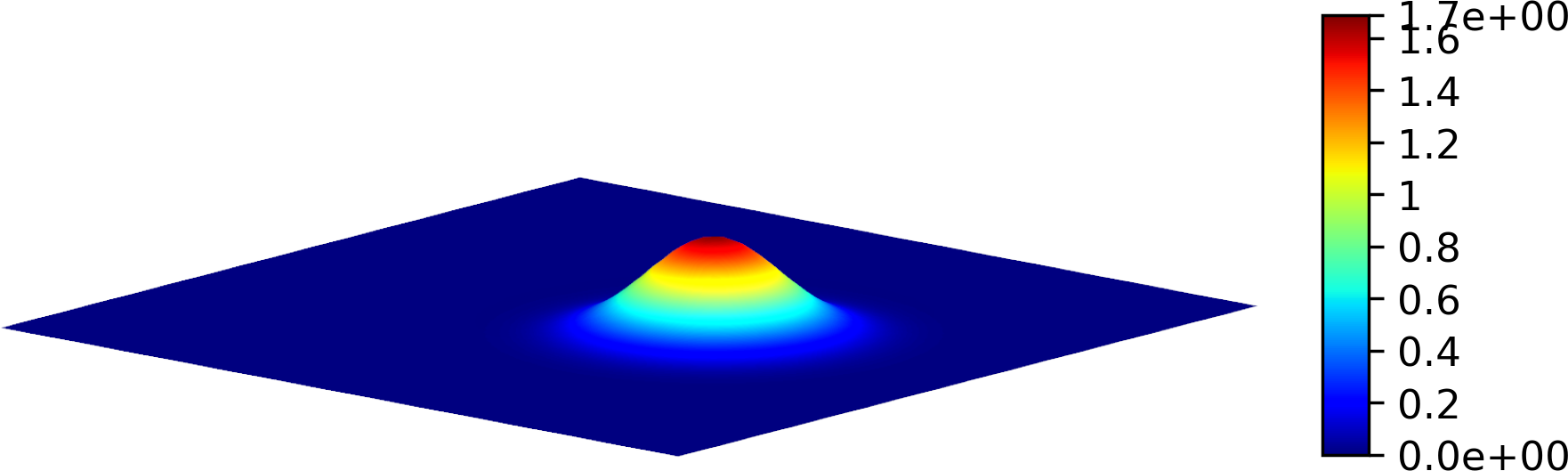} \\
(a) initial, $t=\pi/2$ & (b) final, $t=3\pi/2$
\end{tabularx}
\caption{Rotating Gaussian hill. Snapshots of the initial $t=\pi/2$ (left) and final time $t=3\pi/2$ (right) solutions. Produced by $\IQ_3$ scheme with mesh resolution $h = 1/16$. The solution is smeared by diffusion over the time evolution from $t=\pi/2$ to $t=3\pi/2$.
Note that the left and right panels use different color scales.}
\label{fig:Gaussian_hill}
\end{figure}
\begin{table}[ht!]
\centering
\begin{tabularx}{0.7\linewidth}{@{~~}c@{\hspace{1.0cm}}c@{\hspace{1.0cm}}c@{\hspace{0.75cm}}C@{}C@{}}
\toprule
$k$ & $h$ & $\tau$ & $\norm{U_h^{\Nst} - u(T)}{L^2_h}$ & rate \\
\midrule                   
2 & $2^{-2}$ & $2^{-1}\cdot\pi\cdot10^{-2}$ & $5.892\cdot10^{-2}$ & ---   \\
~ & $2^{-3}$ & $2^{-4}\cdot\pi\cdot10^{-2}$ & $7.929\cdot10^{-3}$ & 2.893 \\
~ & $2^{-4}$ & $2^{-7}\cdot\pi\cdot10^{-2}$ & $9.931\cdot10^{-4}$ & 2.997 \\
~ & $2^{-5}$ & $2^{-10}\cdot\pi\cdot10^{-2}$ & $1.163\cdot10^{-4}$ & 3.094 \\
\midrule                   
3 & $2^{-2}$ & $2^{-1}\cdot\pi\cdot10^{-2}$ & $5.695\cdot10^{-2}$ & ---   \\
~ & $2^{-3}$ & $2^{-5}\cdot\pi\cdot10^{-2}$ & $3.906\cdot10^{-3}$ & 3.866 \\
~ & $2^{-4}$ & $2^{-9}\cdot\pi\cdot10^{-2}$ & $2.457\cdot10^{-4}$ & 3.991 \\
~ & $2^{-5}$ & $2^{-13}\cdot\pi\cdot10^{-2}$ & $1.455\cdot10^{-5}$ & 4.077 \\
\bottomrule
\end{tabularx}
\caption{Rotating Gaussian hill. The $L^2_h$ errors and convergence rates for $\IQ_2$ and $\IQ_3$ schemes.  Here $\tau$ and $h$ are refined together with $\tau \sim h^{k+1}$ and the sequence satisfies the condition in Theorem~\ref{thm:error_bound_conv_diff} with $c = 0.03$. Such a constant always exists for a finite sequence of refinements. More importantly, even if $c$ is taken small, the estimate in Theorem~\ref{thm:error_bound_unconditional} is independent of $c$, so the error still tends to zero.}
\label{tab:Gaussian_hill}
\end{table}

\subsection{Solid body rotation}
The rotation of a slotted cylinder, a conical body, and a hump in a divergence-free velocity field is a widely used benchmark for convection-diffusion solvers. It assesses how well a scheme preserves complex solution structures over long-time integration.
\par
We choose the domain $\Omega = [0,1]^2$ with final time $T = 1$. The divergence-free velocity field is $\vec{b} = \transpose{[\pi(1-2y), \pi(2x-1)]}$, and the diffusion coefficient is $D = 10^{-4}$. 
To introduce the initial condition, let $B_{r_0}(x_0,y_0)$ denote a ball centered at $\vec{x}_0 = \transpose{[x_0,y_0]}$ of radius $r_0$, namely
\begin{align*}
B_{r_0}(x_0,y_0) = \{\vec{x}\in\Omega: (x-x_0)^2 + (y-y_0)^2 \leq r_0^2\}.
\end{align*}
Initially, the slotted cylinder is given by 
\begin{align*}
u^0(\vec{x}) = \begin{cases}
1, & \text{if}~\vec{x}\in B_{0.15}(0.5, 0.75) \setminus \{\vec{x}: \abs{x - 0.5} < 0.025~\text{and}~\abs{y} < 0.85\}, \\
0, & \text{if}~\vec{x}\in B_{0.15}(0.5, 0.75) \cap \{\vec{x}: \abs{x - 0.5} < 0.025~\text{and}~\abs{y} < 0.85\}.
\end{cases}
\end{align*}
The conical body is given by
\begin{align*}
u^0(\vec{x}) = 1 - \sqrt{(x - 0.5)^2 + (y - 0.25)^2}/0.15,\quad \text{if}~\vec{x}\in B_{0.15}(0.5, 0.25).
\end{align*}
The hump is given by
\begin{align*}
u^0(\vec{x}) = 0.25(1 + \cos(\pi\sqrt{(x - 0.25)^2 + (y - 0.5)^2}/0.15)),\quad \text{if}~\vec{x}\in B_{0.15}(0.25, 0.5).
\end{align*}
Elsewhere, $u^0(\vec{x}) = 0$.
In addition, we impose the no-flux boundary condition on $\partial\Omega$.
\par
We use uniform square meshes for both $\IQ_2$ and $\IQ_3$ spaces. In each case, the discrete solution is represented using Lagrangian basis functions, while the integrals are evaluated using Gaussian quadrature. The cell average limiter with $[m,M] = [0,1]$ is triggered during the simulation. 
\par
Figure~\ref{fig:rot_solid_body} shows the initial data and the solutions at the final time for three configurations: $\IQ_2$ with $h=1/160$ and $\tau=0.0002$, $\IQ_2$ with $h=1/200$ and $\tau=0.0001$, and $\IQ_3$ with $h=1/160$ and $\tau=0.0001$. In all cases, the solution remains within the bound $[0,1]$ throughout the simulation. The $\IQ_2$ solution on the coarser mesh exhibits noticeable numerical diffusion. A comparison of the $\IQ_2$ results on $h=1/160$ and $h=1/200$ shows that mesh refinement and smaller time steps reduce numerical diffusion and improve solution quality. The $\IQ_2$ solution on $h=1/200$ and the $\IQ_3$ solution on $h=1/160$ have roughly comparable numbers of degrees of freedom and are visually indistinguishable, indicating that the spatial resolution is sufficiently fine and the temporal error dominates. 
Further refining the mesh or increasing the polynomial order yields diminishing improvements. Because the implicit Euler method is first-order in time, reducing $\tau$ from $0.0002$ to $0.0001$ yields only a modest improvement.
\par
This test validates the bound-preserving property of our numerical scheme. The discrete solution remains in $[0,1]$ and conserves mass over the entire time evolution, consistent with our theoretical analysis.
\begin{figure}[ht!]
\centering
\begin{tabularx}{0.95\linewidth}{@{}c@{\quad}c@{\quad}c@{~}}
\includegraphics[width=0.4\textwidth]{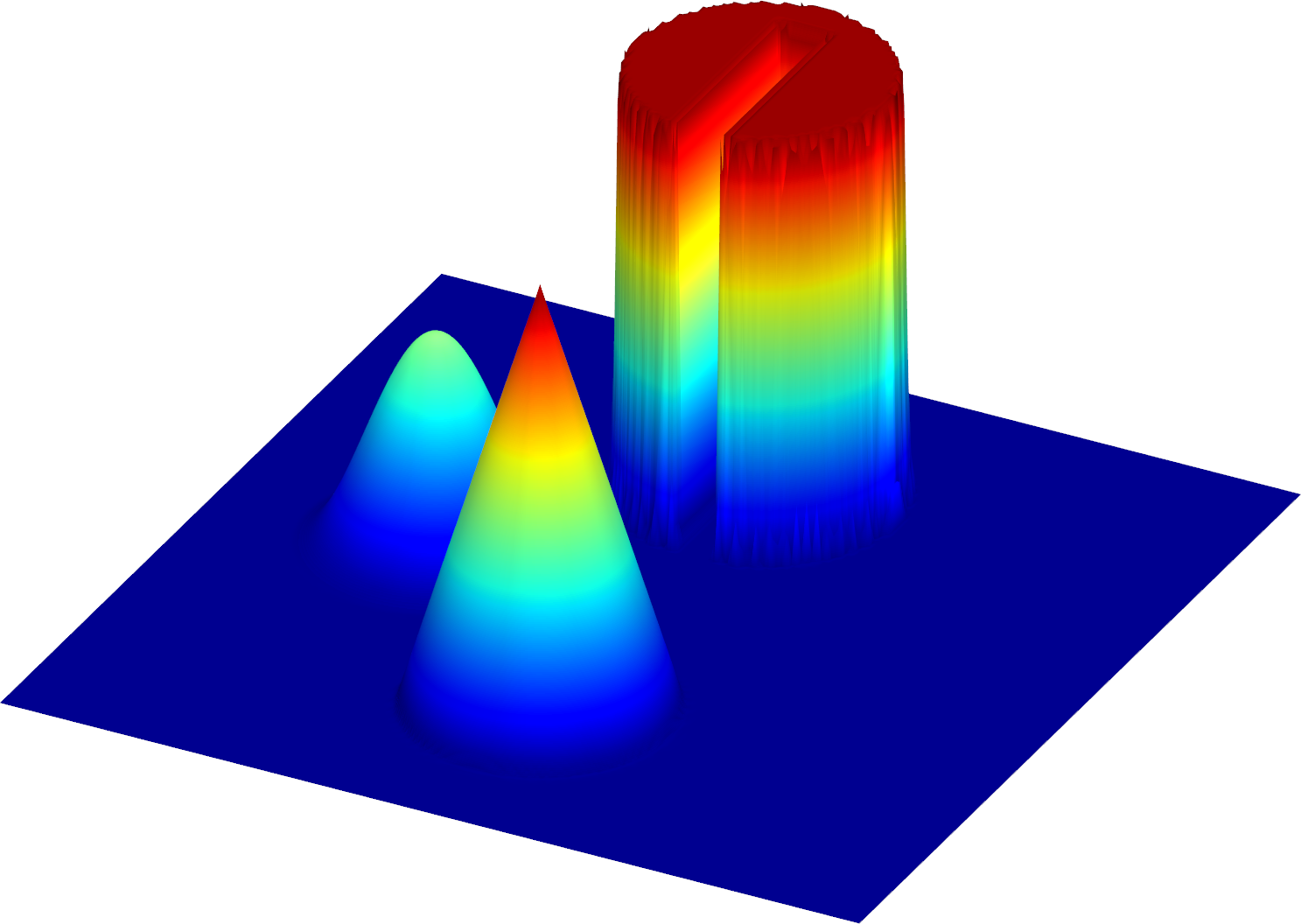} &
\includegraphics[width=0.4\textwidth]{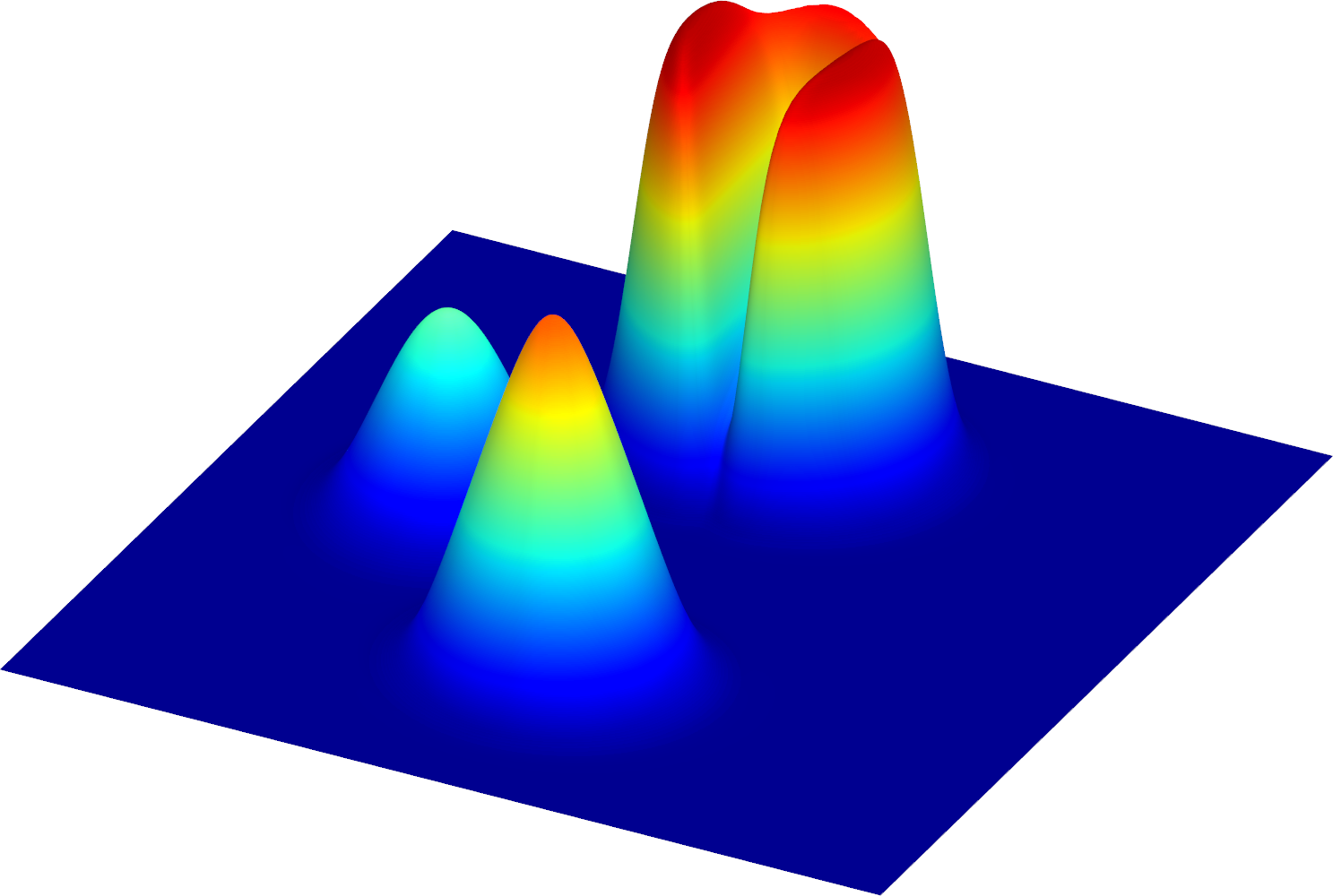} &
\includegraphics[width=0.07\textwidth]{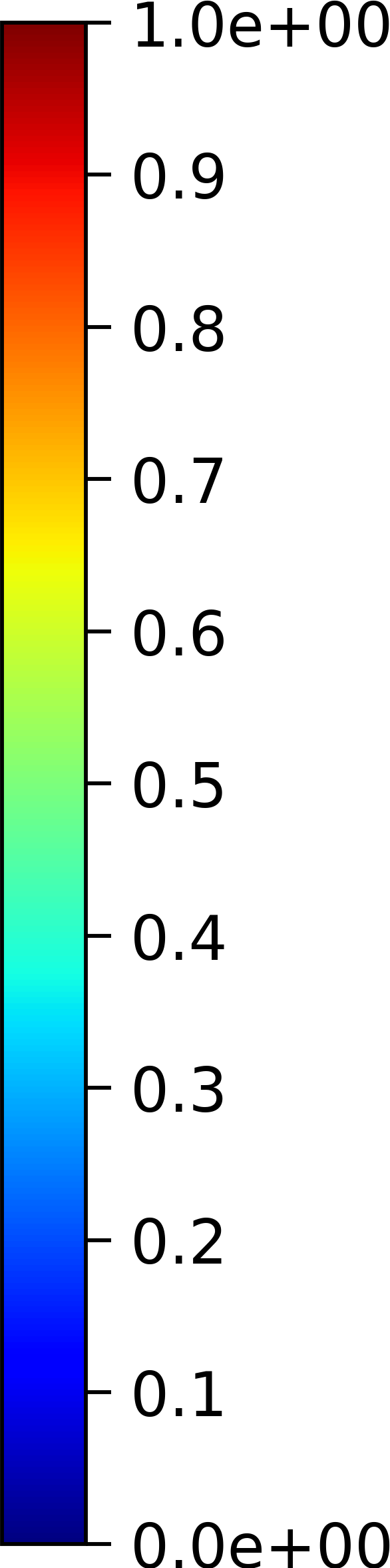} \\
initial & $\IQ_2$ scheme, $h=1/160$, $\tau=0.0002$ & ~ \\
\includegraphics[width=0.4\textwidth]{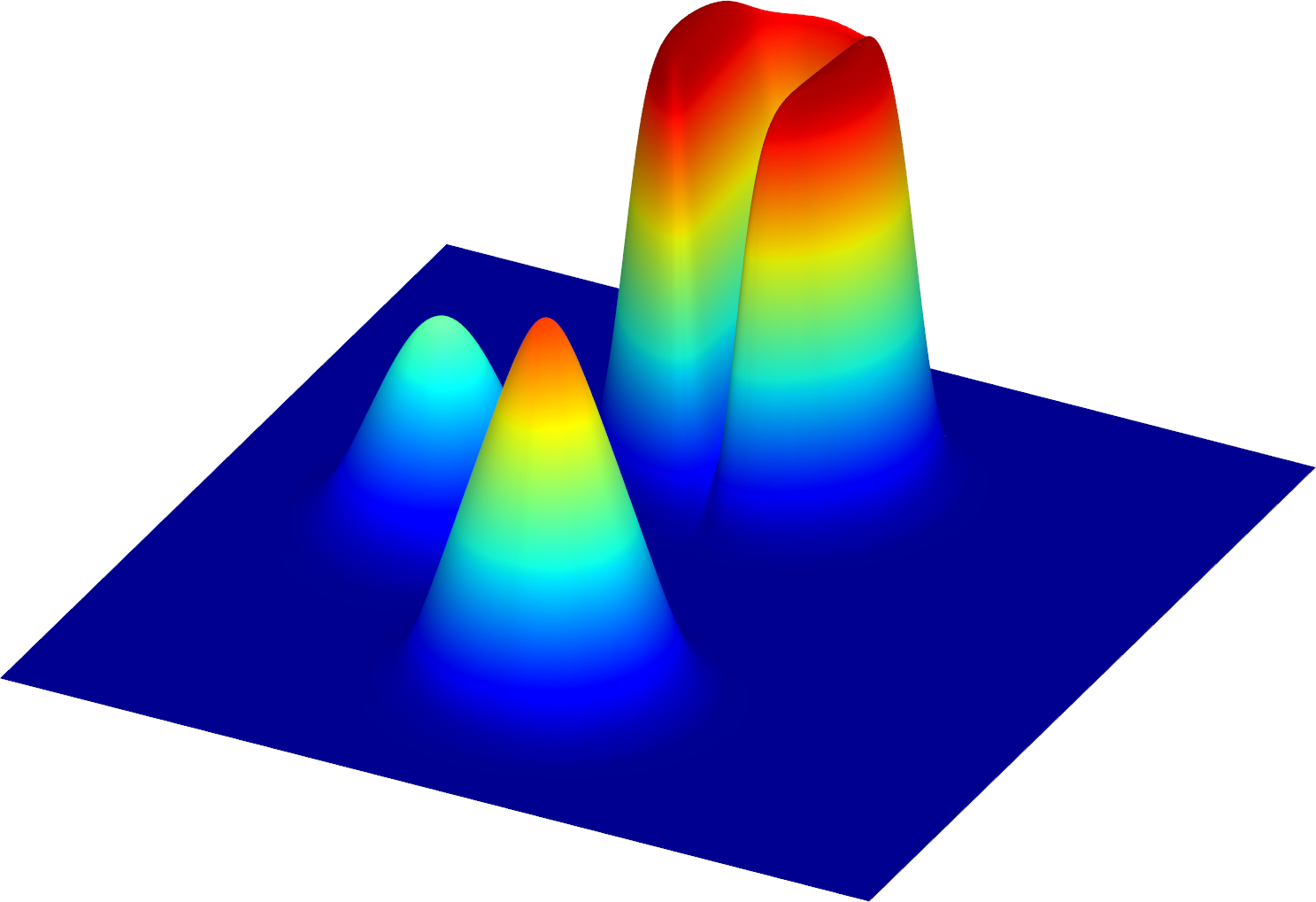} &
\includegraphics[width=0.4\textwidth]{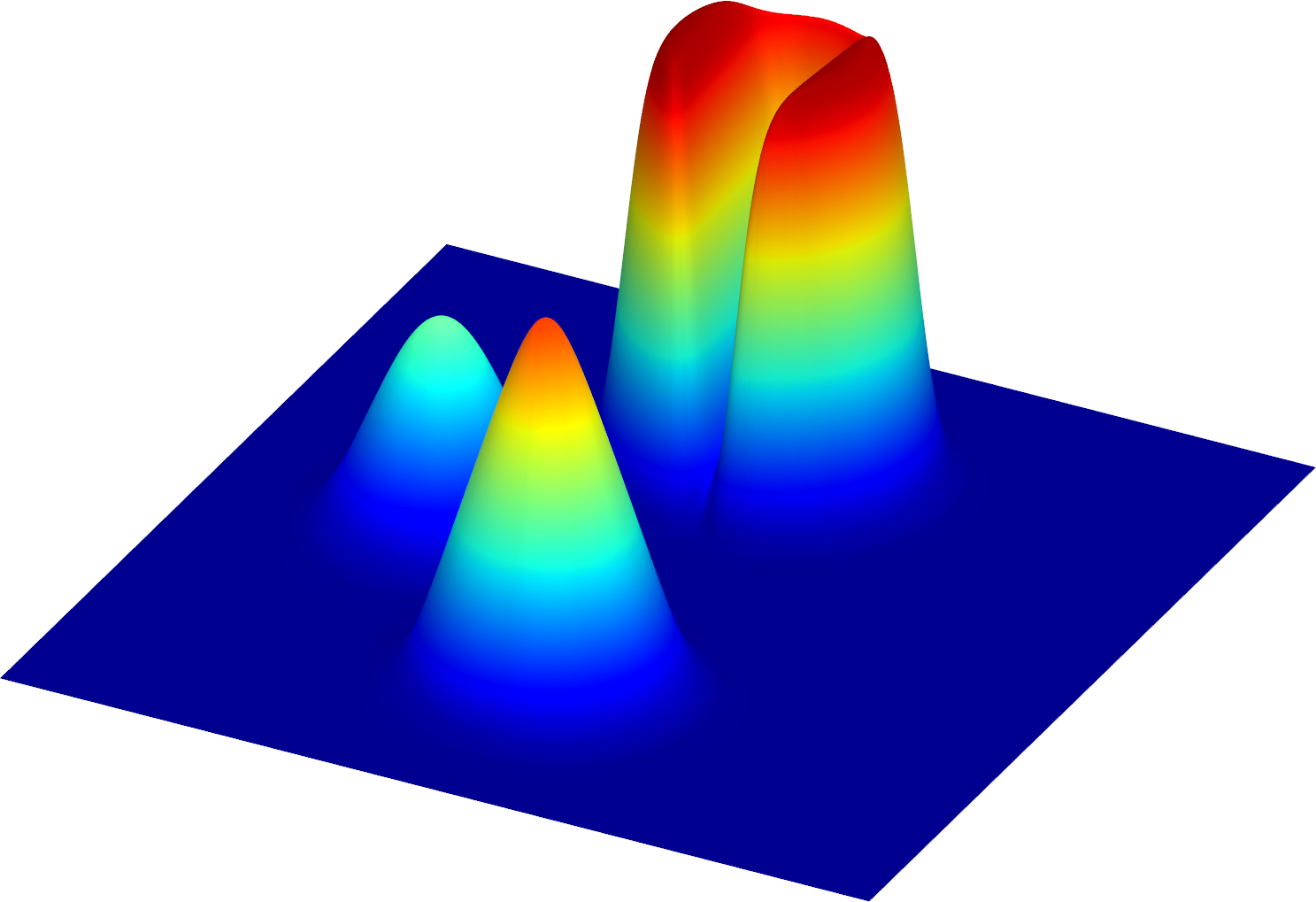} &
\includegraphics[width=0.075\textwidth]{colorbar.png} \\
$\IQ_2$ scheme, $h=1/200$, $\tau=0.0001$ & $\IQ_3$ scheme, $h=1/160$, $\tau=0.0001$ & ~ \\
\end{tabularx}
\caption{Rotating solid bodies. Snapshots of the initial condition and solutions at the final time $T = 1$.}
\label{fig:rot_solid_body}
\end{figure}

\section{Conclusions}\label{sec:conclusion}
In this paper, we proposed a constrained evolution framework that encodes conservation and admissibility as a differential inclusion. Its Lie splitting yields a two-stage scheme: an implicit Euler DG discretization of the underlying PDE, followed by an optimization-based cell average limiter realized as a backward Euler proximal step.
We proved that the limiter is nonexpansive and stable in the $L^2$ norm, and does not increase the total variation of the cell averages. For the heat and convection-diffusion equations, nonexpansiveness yields an $\mathcal{O}(\tau+h^{k+1})$ error estimate for the fully discrete scheme under a reverse CFL condition $\tau\geq ch^{2(k+1)/3}$, together with an unconditional $\mathcal{O}(\tau+h^{k})$ estimate.
Our analysis offers an explanation for why the optimization-based cell average limiter preserves accuracy so robustly in practice: it is the metric projection onto the convex admissible set, and is therefore nonexpansive. 
Numerical experiments on the heat equation and a rotating Gaussian hill confirm the optimal convergence rates while enforcing the prescribed bounds. The solid body rotation test further validates conservation and bound preservation for the full scheme.

\section*{Acknowledgments}
The authors thank Prof.~Xiangxiong Zhang for helpful discussions on optimization-based limiters.
Sanghyun Lee is partially supported by the U.S. National Science Foundation under Grant DMS-2208402 and the U.S. Department of Energy, Office of Science, Energy Earthshots Initiatives under Award Number DE-SC-0024703. The work of Seungmin Lee is supported by the U.S. Department of Energy, Office of Science, Energy Earthshots Initiatives under Award Number DE-SC-0024703.
Chen Liu is partially supported by the U.S. National Science Foundation under Grant DMS-2513106.

\paragraph{Declaration of generative AI and AI-assisted technologies in the
manuscript preparation process}
During the preparation of this work, the authors used Claude (Anthropic) to assist with mathematical discussions and improve the language of the manuscript. After using this tool, the authors carefully reviewed and edited the content as needed and take full responsibility for the content of the published article.

\appendix

\section{Proof of Theorem~\ref{thm:error_bound_unconditional}}\label{app:unconditional_subopt}

Let us derive a bound on the bilinear form $a_\mathcal{D}$ when its first argument is not a broken polynomial. For any $v\in H^2(\setE_h)$, define the augmented norm $\norm{\cdot}{\ast}$ as follows
\begin{align*}
\norm{v}{\ast}^2 = \sum_{E\in\setE_h}\norm{\grad v}{L^2(E)}^2
+ \sum_{e\in\Gammah} h\norm{\avg{\grad v}}{L^2(e)}^2
+ \frac{\sigma}{h}\sum_{e\in\Gammah}\norm{\jump{v}}{L^2(e)}^2.
\end{align*}
\begin{lemma}\label{lem:aD_bound}
There exists a constant $C > 0$ independent of $h$, such that $\forall v\in H^2(\setE_h)$ and $\forall v_h\in V_h$, we have $\abs{a_\mathcal{D}(v, v_h)} \leq C\norm{v}{\ast}\norm{v_h}{\mathrm{DG}}$.
\end{lemma}
\begin{proof}
Recalling the definition of the SIPG bilinear form $a_\mathcal{D}$, we have
\begin{multline*}
a_\mathcal{D}(v, v_h) = \sum_{E\in\setE_h}\int_E \grad v\cdot\grad v_h
- \sum_{e\in\Gammah}\int_e \avg{\grad{v}\cdot\normal_e}\jump{v_h} 
- \sum_{e\in\Gammah}\int_e \avg{\grad{v_h}\cdot\normal_e}\jump{v}
+ \frac{\sigma}{h}\sum_{e\in\Gammah}\int_e \jump{v}\jump{v_h}.
\end{multline*}
For the first, the second, and the last terms above, applying Cauchy--Schwarz's inequality, we obtain
\begin{align*}
\Big|\sum_{E\in\setE_h}\int_E \grad v\cdot\grad v_h\Big|
&\leq \Big(\sum_{E\in\setE_h}\norm{\grad v}{L^2(E)}^2\Big)^{1/2}
\Big(\sum_{E\in\setE_h}\norm{\grad v_h}{L^2(E)}^2\Big)^{1/2}
\leq \norm{v}{\ast}\norm{v_h}{\mathrm{DG}},\\
\Big|-\sum_{e\in\Gammah}\int_e \avg{\grad{v}\cdot\normal_e}\jump{v_h}\Big|
&\leq \Big(\sum_{e\in\Gammah} h\norm{\avg{\grad v}}{L^2(e)}^2\Big)^{1/2}
\Big(\sum_{e\in\Gammah} \frac{1}{h}\norm{\jump{v_h}}{L^2(e)}^2\Big)^{1/2}
\leq \frac{1}{\sqrt{\sigma}}\norm{v}{\ast}\norm{v_h}{\mathrm{DG}},\\
\Big|\frac{\sigma}{h}\sum_{e\in\Gammah}\int_e \jump{v}\jump{v_h}\Big|
&\leq \Big(\frac{\sigma}{h}\sum_{e\in\Gammah}\norm{\jump{v}}{L^2(e)}^2\Big)^{1/2}
\Big(\frac{\sigma}{h}\sum_{e\in\Gammah}\norm{\jump{v_h}}{L^2(e)}^2\Big)^{1/2}
\leq \norm{v}{\ast}\norm{v_h}{\mathrm{DG}}.
\end{align*}
For the third term on the right-hand side of $a_\mathcal{D}(v, v_h)$, we apply the discrete trace inequality and obtain
\begin{align*}
\Big|-\sum_{e\in\Gammah}\int_e \avg{\grad{v_h}\cdot\normal_e}\jump{v}\Big|
&\leq \Big(\sum_{e\in\Gammah} h\norm{\avg{\grad v_h}}{L^2(e)}^2\Big)^{1/2}
\Big(\sum_{e\in\Gammah} \frac{1}{h}\norm{\jump{v}}{L^2(e)}^2\Big)^{1/2} \\
&\leq \frac{C}{\sqrt{\sigma}}\Big(\sum_{E\in\setE_h}\norm{\grad v_h}{L^2(E)}^2\Big)^{1/2}\norm{v}{\ast}
\leq C\norm{v}{\ast}\norm{v_h}{\mathrm{DG}}.
\end{align*}
Combining the four bounds above, we conclude the proof.
\end{proof}

Let us define the approximation error $\rho^n = u^n - \pi_h u^n$, and the discrete errors $\varphi_h^n = u_h^n - \pi_h u^n$ and $\Phi_h^n = U_h^n - \pi_h u^n$. We have the following error equation
\begin{multline}\label{eq:phi_eq}
(\varphi_h^{n+1} - \Phi_h^n, v_h) + \tau D a_\mathcal{D}(\varphi_h^{n+1}, v_h)
+ \tau a_{\mathrm{adv}}(\varphi_h^{n+1}, v_h)
= (\rho^{n+1} - \rho^n, v_h) \\
+ \tau D a_\mathcal{D}(\rho^{n+1}, v_h) + \tau a_{\mathrm{adv}}(\rho^{n+1}, v_h)
+ \int_{t^n}^{t^{n+1}} (t-t^n)(\partial_{tt} u, v_h), \qquad \forall v_h\in V_h.
\end{multline}
The $L^2$ projection preserves cell averages, so $\overline{\pi_h u^n} = \overline{u^n}$, which lies in $\Lambda$ by Theorem~\ref{thm:mass_conservation}. Thus, $\mathcal{Q}\overline{\pi_h u^n} = \overline{\pi_h u^n}$ and $\mathcal{P}_h \pi_h u^n = \pi_h u^n$.
Recalling $U_h^n = \mathcal{P}_h u_h^n$ and applying Theorem~\ref{thm:P_nonexpansive}, we have
\begin{align}\label{eq:Phi_bound}
\norm{\Phi_h^n}{L^2} = \norm{\mathcal{P}_h u_h^n - \mathcal{P}_h \pi_h u^n}{L^2} \leq \norm{u_h^n - \pi_h u^n}{L^2} = \norm{\varphi_h^n}{L^2}.
\end{align}
Choosing $v_h = \varphi_h^{n+1}\in V_h^0$ in \eqref{eq:phi_eq}, by the coercivity of $a_\mathcal{D}$, the positivity of $a_{\mathrm{adv}}$, the definition of $\pi_h$ which gives $(\rho^{n+1} - \rho^n, \varphi_h^{n+1}) = 0$, and \eqref{eq:Phi_bound}, we have
\begin{multline}\label{eq:after_sub_uncond}
\frac{1}{2}\norm{\varphi_h^{n+1}}{L^2}^2 + \tau D\kappa\norm{\varphi_h^{n+1}}{\mathrm{DG}}^2
\leq \frac{1}{2}\norm{\varphi_h^{n}}{L^2}^2 + \tau D a_\mathcal{D}(\rho^{n+1}, \varphi_h^{n+1}) \\
+ \tau a_{\mathrm{adv}}(\rho^{n+1}, \varphi_h^{n+1}) + \int_{t^n}^{t^{n+1}} (t-t^n)(\partial_{tt} u, \varphi_h^{n+1}).
\end{multline}
For the second term on the right-hand side of \eqref{eq:after_sub_uncond}, using Lemma~\ref{lem:aD_bound}, the approximation property of the $L^2$ projection \eqref{eq:L2_proj_approx}, and the trace inequality, we have
\begin{align}\label{eq:uncond_bound_rhs1}
&\abs{a_\mathcal{D}(\rho^{n+1}, \varphi_h^{n+1})} 
\leq C\norm{\rho^{n+1}}{\ast}\norm{\varphi_h^{n+1}}{\mathrm{DG}} \nonumber\\
=\,& C\Big(\sum_{E\in\setE_h}\norm{\grad{\rho^{n+1}}}{L^2(E)}^2 + \sum_{e\in\Gammah} h\norm{\avg{\grad{\rho^{n+1}}}}{L^2(e)}^2 + \frac{\sigma}{h}\sum_{e\in\Gammah}\norm{\jump{\rho^{n+1}}}{L^2(e)}^2\Big)^{1/2} \norm{\varphi_h^{n+1}}{\mathrm{DG}} \nonumber\\
\leq\,& C h^{k}\abs{u^{n+1}}_{H^{k+1}} \norm{\varphi_h^{n+1}}{\mathrm{DG}}
\leq \frac{\kappa}{4}\norm{\varphi_h^{n+1}}{\mathrm{DG}}^2 + \frac{C}{\kappa}h^{2k}\abs{u^{n+1}}_{H^{k+1}}^2.
\end{align}
The third term on the right-hand side of \eqref{eq:after_sub_uncond} is bounded exactly as \eqref{eq:conv_diff_bound_rhs2}, with $\zeta^{n+1}$ replaced by $\rho^{n+1}$, and the fourth term exactly as \eqref{eq:conv_diff_bound_rhs3}, with $\theta_h^{n+1}$ replaced by $\varphi_h^{n+1}$. Using the approximation property of the $L^2$ projection \eqref{eq:L2_proj_approx}, we have
\begin{multline}\label{eq:uncond_bound_rhs23}
\tau\abs{a_{\mathrm{adv}}(\rho^{n+1}, \varphi_h^{n+1})} + \Big|\int_{t^n}^{t^{n+1}}(t-t^n)(\partial_{tt}u, \varphi_h^{n+1})\Big| \\
\leq \frac{\tau D\kappa}{2}\norm{\varphi_h^{n+1}}{\mathrm{DG}}^2 + \frac{C\tau h^{2k+2}}{D\kappa} + \frac{C\tau^2}{D\kappa}\norm{\partial_{tt}u}{L^2(t^n,t^{n+1}; L^2)}^2.
\end{multline}
Substituting \eqref{eq:uncond_bound_rhs1} and \eqref{eq:uncond_bound_rhs23} into \eqref{eq:after_sub_uncond}, and noting that $u_h^0 = \pi_h u^0$ gives $\varphi_h^0 = 0$, for any integer $1\leq N\leq\Nst$, summing over $n = 0,\cdots,N-1$, after some manipulations, we have
\begin{align}\label{eq:L2_bound_phi}
\norm{\varphi_h^{N}}{L^2}^2 + \frac{\tau D\kappa}{2 }\sum_{n=0}^{N-1}\norm{\varphi_h^{n+1}}{\mathrm{DG}}^2
\leq \frac{CTD}{\kappa}h^{2k}\norm{u}{L^\infty(0,T; H^{k+1})}^2 + \frac{CT}{D\kappa}h^{2k+2} + \frac{C\tau^2}{D\kappa}\norm{\partial_{tt}u}{L^2(0,T; L^2)}^2,
\end{align}
and therefore $\max_n\norm{\varphi_h^{n}}{L^2} \leq C(\tau + h^{k})$. 
Finally, by the error decomposition $u^n - U_h^n = \rho^n - \Phi_h^n$, \eqref{eq:Phi_bound} and the approximation property of $\pi_h$, we have
\begin{align*}
\norm{u^n - U_h^n}{L^2}
\leq \norm{\rho^n}{L^2} + \norm{\Phi_h^n}{L^2}
\leq Ch^{k+1}\abs{u^n}_{H^{k+1}} + \norm{\varphi_h^n}{L^2}
\leq C(\tau + h^{k}).
\end{align*}
Taking the maximum over $n$ concludes the proof.

\printbibliography

\end{document}